\documentclass[12pt]{article}
\usepackage{amsmath,amssymb,amsthm,bbm}
\usepackage[colorlinks=true,citecolor=blue,urlcolor=blue,linkcolor=blue]{hyperref}
\usepackage[margin=1in]{geometry}
\usepackage[round]{natbib}
\usepackage{color}
\usepackage{graphicx}
\usepackage{stmaryrd} 
\usepackage{microtype}
\usepackage{enumitem}
\usepackage{ragged2e}

\newtheorem{theorem}{Theorem}
\newtheorem{lemma}{Lemma}
\newtheorem{proposition}{Proposition}
\newtheorem{corollary}{Corollary}

\theoremstyle{definition}
\newtheorem{remark}{Remark}

\newcommand{\R}{\mathbb{R}}

\newcommand{\EE}[1]{\mathbb{E}\left[{#1}\right]}

\newcommand{\Ep}[2]{\mathbb{E}_{{#1}}\left[{#2}\right]}
\newcommand{\Epst}[3]{\mathbb{E}_{{#1}}\left[{#2}\ \middle| \ {#3}\right]}
\newcommand{\PP}[1]{\mathbb{P}\left\{{#1}\right\}}
\newcommand{\PPst}[2]{\mathbb{P}\left\{{#1}\ \middle| \ {#2}\right\}}
\newcommand{\Ppst}[3]{\mathbb{P}_{{#1}}\left\{{#2}\ \middle| \ {#3}\right\}}
\newcommand{\Pp}[2]{\mathbb{P}_{{#1}}\left\{{#2}\right\}}
\newcommand{\eqd}{\stackrel{\textnormal{d}}{=}}
\newcommand{\One}[1]{{\mathbbm{1}}\left\{{#1}\right\}}

\newcommand{\PPin}[1]{\mathbb{P}\{{#1}\}} 
\newcommand{\iidsim}{\stackrel{\textnormal{iid}}{\sim}}

\newcommand\independent{\protect\mathpalette{\protect\independenT}{\perp}}
\def\independenT#1#2{\mathrel{\rlap{$#1#2$}\mkern2mu{#1#2}}}

\newcommand{\Var}{\mathrm{Var}}
\newcommand{\dtv}{\mathrm{d}_{\mathrm{TV}}}
\newcommand{\Quantile}{\mathrm{Quantile}}
\def\Xcal{\mathcal{X}}
\def\Ycal{\mathcal{Y}}
\def\Ccal{\mathcal{C}}
\def\Scal{\mathcal{S}}
\def\Zcal{\mathcal{Z}}
\def\Ucal{\mathcal{U}}
\def\Zbag{\lbag Z \rbag}
\def\zbag{\lbag z \rbag}
\def\thresh{\mathsf{t}}
\def\score{\mathsf{s}}
\def\teststat{\varphi}
\def\alg{\mathcal{A}}
\def\Ptr{F} 
\def\Pte{G} 
\def\d{\mathsf{d}}
\def\id{\mathsf{id}}

\title{Unifying Different Theories of Conformal Prediction}
\author{Rina Foygel Barber\thanks{Department of Statistics, University of
    Chicago} \and Ryan J.\ Tibshirani\thanks{Department of Statistics,
    University of California, Berkeley}} 
\date{}

\begin{document}
\maketitle

\begin{abstract}
This paper establishes a unified framework for understanding the methodology and 
theory behind several different methods in the conformal prediction literature,
which includes standard conformal prediction (CP), weighted conformal prediction
(WCP), nonexchangeable conformal prediction (NexCP), and randomly-localized
conformal prediction (RLCP), among others. At the crux of our framework is the
idea that conformal methods are based on revealing \emph{partial information}
about the data at hand, and positing a conditional distribution for the data
given the partial information. Different methods arise from different choices of
partial information, and of the corresponding (approximate) conditional
distribution. In addition to recovering and unifying existing results, our
framework leads to both new theoretical guarantees for existing methods, and new
extensions of the conformal methodology.
\end{abstract}

\section{Introduction}
\label{sec:intro}

As machine learning algorithms become increasingly embedded in prediction
systems, it has become increasingly important to address a question of
reliability: how can we quantify the uncertainty in the predictions that are
produced by black-box models? Conformal prediction \citep{vovk2005algorithmic}
is a framework for providing predictive inference around the output of a
black-box algorithm, offering a guarantee of predictive coverage that relies 
only on assuming the data points are exchangeable, without placing further
assumptions on the distribution of the data or any assumptions on the algorithm
used for modeling.

To be more precise, conformal prediction works in a setting in which we observe
training data $(X_1,Y_1),\dots,(X_n,Y_n)\in\Xcal\times\Ycal$, where each $X_i$
is a feature (e.g., a vector of covariates) and each $Y_i$ is a response. We
also have a test point $(X_{n+1},Y_{n+1})$, where the feature $X_{n+1}$ is
observed, and our task is to predict the unobserved response $Y_{n+1}$. Many
different methods from statistics and machine learning can be applied to produce
a fitted model $\hat{f}:\Xcal\to\Ycal$ to (often accurately) predict $Y$ from
$X$. How can we build a prediction interval around $\hat{f}(X_{n+1})$ to
quantify our uncertainty about $Y_{n+1}$? To answer this question, conformal
prediction relies on the assumption that the data is exchangeable, i.e., the
assumption that the distribution of $((X_1,Y_1),\dots,(X_{n+1},Y_{n+1}))$ is
invariant to permuting these $n+1$ data points. Section \ref{sec:background}
gives more background on this method.

In recent years, the conformal prediction literature has seen a flurry of
development. Our focus in this paper is on extensions of the conformal
framework which permit relaxations of the core assumption of exchangeability, 
and advances around localization. This includes: 

\begin{itemize}
\item Weighted conformal prediction \citep{tibshirani2019conformal}, which
  applies likelihood-based weights to accommodate settings where the training 
  and test data come from different distributions, with a known shift. 
  \citet{tibshirani2019conformal} focus on the case of covariate shift, while 
  \citet{podkopaev2021distribution} study the problem of label shift.   

\item Nonexchangeable conformal prediction \citep{barber2023conformal}, which
  applies a different sort of weighting to improve the robustness of the
  prediction sets to mild but unknown violations of the exchangeability
  assumption.         

\item Localized conformal prediction \citep{guan2023localized} and
  randomly-localized conformal prediction \citep{hore2023conformal}, which
  modify the method to place higher weight on data points closer to the test 
  point $X_{n+1}$, in order to aim for better locally adaptivity, i.e.,
  coverage which approximately holds once we condition on $X_{n+1}$.  

\item Generalized weighted conformal prediction \citep{prinster2024conformal},
  which offers a generalized view through the lens of understanding
  distributions of all possible permutations of the data, to handle arbitrary
  forms of nonexchangeability; earlier work by \citet{fannjiang2022conformal}
  treats an important special case arising in experimental design.
\end{itemize}

\paragraph{Our contributions.}

In this paper, we develop a unified theory that brings together these methods 
and their supporting theory. At a high level, this unified view follows a simple
but general recipe: given a dataset $Z = ((X_1,Y_1),\dots,(X_{n+1},Y_{n+1}))$,
we suppose that we observe partial information $U$ about $Z$. The construction
of a conformal prediction set then relies on calculating the (exact or
approximate) distribution of the test point $(X_{n+1},Y_{n+1})$ conditional on
the partial information $U$. For example, in standard conformal prediction, we
define $U$ to be the unordered collection of the data points in $Z$, and the
test point $(X_{n+1},Y_{n+1})$ is equally likely to be any one of the $n+1$ data
points in this collection, due to exchangeability.

We will see that this unified view recreates the existing extensions of standard 
conformal prediction highlighted above, offering a shared theory that explains 
these seemingly disparate generalizations which have been discovered in the 
recent literature. We will also see that the unified view allows us to fluidly
derive new results, ultimately leading to broader applicability of the conformal
prediction framework.       

\paragraph{Additional related work.}

The goal of the current work is to unify several extensions of conformal 
prediction under departures from exchangeability. Meanwhile, various other new 
developments have been recently made in the literature on conformal prediction. 
We survey some of these advances here. 

While conformal prediction can be paired with any modeling algorithm (to make
predictions), and any score function (to measure the predictive model's
accuracy), the utility of the method relies on choosing this pairing
appropriately for the problem at hand. There is by now a rich literature
examining different choices, and we highlight just a few contributions. For
the setting of a real-valued response $Y$, popular options include the residual
score and scaled residual score methods studied by \citet{lei2018distribution},
the quantile score studied by \citet{romano2019conformalized}, a method for
identifying regions of high density by \citet{izbicki2022cd}, and a
nearest-neighbor method studied by \citet{gyorfi2019nearest}. For the setting of
a categorical response, various novel score functions have also been proposed
and studied by \citet{sadinle2019least, romano2020classification,
  angelopoulos2021uncertainty}, among others.

Other lines of work focus on developing methodology and/or theory in different
contexts. We highlight two such lines. First, if the data in fact follows a
specified model (which can be known exactly or approximately), then it is
possible to derive theory which ensures that the conformal prediction method is
competitive with model-based methods; see, e.g., Chapter 5 of
\citet{angelopoulos2024theoretical} for an overview of results of this type.
Second, in the presence of arbitrary (possibly large and discontinuous) and
unknown distribution shift, we can no longer rely on any of the extensions
described above which consider deviations from exchangeability.  In this
setting, the work of \citet{gibbs2021adaptive} proposes an online variant of
conformal prediction that offers a weaker guarantee---coverage is guaranteed on
average over the stream of data. These ideas have been further extended by
\citet{gibbs2024conformal, zaffran2022adaptive, bhatnagar2023improved,
  angelopoulos2023conformal, angelopoulos2024online}, among others.

Finally, we are not the only authors to pursue a unifying framework for
inference which encompasses standard conformal prediction as a special case. A
notable example is the work on online compression models by \citet{vovk2003well,
vovk2023power}. Here a ``summary statistic'' is tracked online, and is updated 
sequentially each time a new data point arrives. This is paired with a
``backward kernel'', which reconstructs the conditional distribution of the data
observed thus far, given the summary statistic. For exchangeable data,
the summary statistic can be set to the bag (unordered multiset of data
points), and this model reproduces conformal prediction in the online setting. 
However, the framework is broader and encapsulates certain Markov or
hypergraphical models as well. 

An online compression model is similar to our framework, as defined below in
Section \ref{sec:unified}, with the summary statistic analogous to our partial
information $U$, and the backward kernel analogous to our conditional
distribution $Q_{Z\mid U}$. Meanwhile, an online compression model is more
restricted than our framework in certain ways, and focused on complementary
aspects. The summary statistic is a deterministic function of the data, and
ideally, one which allows for efficient online updates. Furthermore, the
backward kernel must be exact. In comparison, we do not require $U$ to be a
deterministic function of $Z$ (permitting additional randomness, which is
important for certain applications, such as localized conformal prediction). We
also do not require $Q_{Z\mid U}$ to be correct, developing robustness results
that are at the core of our theory. A reflection on how our work relates to the
broader lanscape of work on conditional inference, in both a classical and
modern sense, is given the discussion section. 

\section{Background}
\label{sec:background}

In this section, we review background material on (standard) conformal
prediction. We also review an alternative formulation of conformal prediction
using the language of hypothesis testing, which will be used to construct our
unified framework in the following sections. For more background on the ideas
explored in this section see, e.g., \citet{vovk2005algorithmic,
  lei2018distribution, angelopoulos2024theoretical}. 

\subsection{Conformal prediction for exchangeable data}
\label{sec:cp_background}

Given training data $(X_1,Y_1),\dots,(X_n,Y_n)\in\Xcal\times\Ycal$, and a test
covariate $X_{n+1}\in\Xcal$, suppose we would like to predict the corresponding
response value $Y_{n+1}\in\Ycal$. Conformal prediction is a framework for
producing a prediction interval, or more generally, a prediction set
$\Ccal(X_{n+1})$ that aims to contain the test response $Y_{n+1}$ with some
prescribed coverage probability $1-\alpha$. To run conformal prediction, we need
to specify a score function:
\[
\score : (\Xcal\times\Ycal) \times (\Xcal\times\Ycal)^{n+1} \to \R,
\]
which compares a data point $(x,y)$ to a given dataset of points
$z_i=(x_i,y_i)$, $i=1,\dots,n+1$. Larger values of the score indicate that
$(x,y)$ does not ``conform'' to the trends observed in the dataset. A canonical
example is any score function of the form
\begin{equation}
\label{eqn:loss_score}
\score\big((x,y),(z_1,\dots,z_{n+1})\big) = \ell(y,\hat{f}(x)), \quad
\text{for $\hat{f} = \alg\big(z_1,\dots,z_{n+1}\big)$},
\end{equation}
where $\ell$ is a loss function on $\Ycal \times \Ycal$, and $\alg$ is a
regression algorithm (e.g., we could use least squares regression in the
real-valued case, $\Ycal = \R$) which inputs a dataset and outputs a fitted 
model $\hat{f} : \Xcal\to\Ycal$.  

Given the score function $\score$, denoting each training point by
$Z_i=(X_i,Y_i)$, we define for arbitrary $y \in \Ycal$ an augmented dataset 
$Z^y = (Z_1,\dots,Z_n,(X_{n+1},y))$, and define scores   
\begin{equation}
\label{eqn:scores}
S^y_i = \score(Z^y_i, Z^y), \quad i \in [n+1].
\end{equation}
where here and throughout we write $[N] = \{1,\dots,N\}$ for an integer $N \geq   
1$. The conformal prediction (CP) set for $Y_{n+1}$ is then given by
\begin{equation}
\label{eqn:conformal_set}
\Ccal(X_{n+1}) = \Big\{ y\in\Ycal : S^y_{n+1} \leq
\Quantile_{1-\alpha}(S^y_1,\dots,S^y_{n+1}) \Big\},
\end{equation}
where $\Quantile_\tau(P) = \inf\{x\in\R: \PP{V \leq x} \geq \tau\}$ denotes the
level-$\tau$ quantile of a random variable $V \sim P$, and we use
$\Quantile_\tau(v_1,\dots,v_N) = \Quantile_\tau(\frac{1}{N} \sum_{i=1}^N
\delta_{v_i})$ to abbreviate the quantile of the empirical distribution
associated with a vector $v\in\R^N$. 

Next we state a well-known, finite-sample guarantee underlying conformal 
prediction. 

\begin{theorem}[{\citealt{vovk2005algorithmic}}]
\label{thm:cp_background}
Suppose $Z_1,\dots,Z_{n+1}$ are exchangeable, and the score function $\score$ is 
\emph{symmetric} in its second argument: 
\begin{equation}
\label{eqn:symmetry}
\score\big((x,y),(z_1,\dots,z_{n+1})\big) =
\score\big((x,y),(z_{\sigma(1)},\dots,z_{\sigma(n+1)})\big), 
\quad \text{for all $\sigma\in\Scal_{n+1}$},  
\end{equation}
where $\Scal_{n+1}$ is the set of permutations on $[n+1]$. Then the 
CP set in \eqref{eqn:conformal_set} satisfies  
\[
\PP{Y_{n+1}\in\Ccal(X_{n+1})} \geq 1-\alpha.
\]
\end{theorem}

To take a closer look at this symmetry condition on $\score$, in the example of
the loss-based score function as in \eqref{eqn:loss_score}, this is equivalent
to requiring the algorithm $\alg$ to be symmetric in the training
data---that is, the function $\hat{f} = \alg(z_1,\dots,z_{n+1})$ is unchanged if
we permute the data points $z_1,\dots,z_{n+1}$ in the training set.   

\paragraph{The bag of data.}
 
For a given $z=(z_1,\dots,z_{n+1})$, we write $\zbag = \lbag z_1,\dots,z_{n+1}
\rbag$ to denote the unordered ``bag'' of data, that is, the (unordered)
multiset obtained from the (ordered) vector $z\in\Zcal^{n+1}$, where $\Zcal =
\Xcal \times \Ycal$. For example, if $z=(1,2,1,5)$ then $\zbag$ conveys that
the dataset contains two $1$s, one $2$, and one $5$, but does not specify the
order in which these four values appear. Since Theorem \ref{thm:cp_background}
assumes $\score((x,y),(z_1,\dots,z_{n+1}))$ is invariant to permutations of the
values $z_1,\dots,z_{n+1}$, we can write the score function under this condition
as
\[
\score((x,y), \lbag z_1,\dots,z_{n+1}\rbag),
\]
where we have overloaded notation, in writing the second argument as a bag of
data. 

\paragraph{Validity of CP.}

With the bag notation in place, we next review the proof of Theorem
\ref{thm:cp_background}. Though there are by now several different ways of
writing the proof, the key intuition is as follows: under exchangeability,
conditional on observing the bag $\Zbag = \lbag Z_1,\dots,Z_n,Z_{n+1}\rbag$, the
test point $Z_{n+1}=(X_{n+1},Y_{n+1})$ is equally likely to be any one of the 
$n+1$ elements.

\begin{proof}[Proof of Theorem \ref{thm:cp_background}] 
Write $S_i = \score(Z_i, \lbag Z_1,\dots,Z_{n+1}\rbag)$, for $i \in
[n+1]$. Recalling \eqref{eqn:scores}, we can observe that $S_i = S_i^{Y_{n+1}}$
(taking $y=Y_{n+1}$). Hence, by construction of the CP set
\eqref{eqn:conformal_set},        
\[
Y_{n+1}\in\Ccal(X_{n+1}) \iff S_{n+1} \leq
\Quantile_{1-\alpha}(S_1,\dots,S_{n+1}), 
\]
thus we need to prove that $\PP{S_{n+1} \leq \Quantile_{1-\alpha}
  (S_1,\dots,S_{n+1})} \geq 1-\alpha$.  By definition of the quantile, we must
have $\frac{1}{n+1} \sum_{i=1}^{n+1} \One{S_i\leq \Quantile_{1-\alpha} 
  (S_1,\dots,S_{n+1})} \geq 1-\alpha$, and so  
\[
\frac{1}{n+1} \sum_{i=1}^{n+1} \PP{S_i\leq
  \Quantile_{1-\alpha}(S_1,\dots,S_{n+1})} \geq 1-\alpha,
\] 
after taking an expectation. Our last step, then, is to verify that for each $i
\in [n+1]$,
\[
\PP{S_{n+1} \leq \Quantile_{1-\alpha}(S_1,\dots,S_{n+1})} = \PP{S_i \leq
  \Quantile_{1-\alpha}(S_1,\dots,S_{n+1})},
\] 
We will prove this by showing that the vector of scores is exchangeable, i.e.,
for all $\sigma\in\Scal_{n+1}$, 
\[
(S_1,\dots,S_{n+1}) \eqd (S_{\sigma(1)},\dots,S_{\sigma(n+1)}).
\]
This is a consequence of the exchangeability of the data, together with the
assumption of symmetry of $\score$. To be more precise, define a function
$f:\Zcal^{n+1}\to\R^{n+1}$ by
\[
f(z) = \big(\score(z_1,\zbag),\dots,\score(z_{n+1},\zbag)\big).
\]
Because the bag of data does not change when the data points are permuted, we
can observe that $f$ commutes with permutations---that is, for any $\sigma \in
\Scal_{n+1}$,  
\[
(s_1,\dots,s_{n+1}) = f(z_1,\dots,z_{n+1}) \iff
(s_{\sigma(1)},\dots,s_{\sigma(n+1)}) = f(z_{\sigma(1)},\dots,z_{\sigma(n+1)}).
\] 
We therefore have for any $\sigma \in \Scal_{n+1}$,
\[
(S_1,\dots,S_{n+1}) = f(Z_1,\dots,Z_{n+1}) \eqd
f(Z_{\sigma(1)},\dots,Z_{\sigma(n+1)}) = (S_{\sigma(1)},\dots,S_{\sigma(n+1)}), 
\] 
where the first step holds by definition of the scores, the second step holds by
exchangeability of the data, and the last step holds since $f$ commutes with
permutations. 
\end{proof}

\subsection{Reframing conformal via hypothesis testing}
\label{sec:conformal_pvalue}

CP provides inference in the form of a prediction set, but the idea can be
equivalently cast in the language of p-values and hypothesis testing. To develop
this equivalence, we define what is known as a conformal p-value: for arbitrary
$y\in\Ycal$, this is    
\begin{equation}
\label{eqn:conformal_pvalue}
p(y) = \frac{\sum_{i=1}^{n+1} \One{S^y_i \geq S^y_{n+1}}}{n+1},
\end{equation}
for $S^y_1,\dots,S^y_{n+1}$ as defined in \eqref{eqn:scores}. Informally, $p(y)$
is often viewed as a p-value for testing the hypothesis $Y_{n+1}=y$, given
exchangeable $(X_1,Y_1),\dots,(X_n,Y_n),(X_{n+1},Y_{n+1})$. The following result 
explains the connection between conformal p-values and prediction sets. 

\begin{proposition}[\citealt{vovk2005algorithmic}]
\label{prop:conformal_pvalue}
The conformal prediction set $\Ccal(X_{n+1})$ defined in
\eqref{eqn:conformal_set} can be written in terms of the conformal p-value
$p(y)$ defined in \eqref{eqn:conformal_pvalue}, as  
\[
\Ccal(X_{n+1}) = \big\{ y\in\Ycal: p(y)>\alpha \big\}.
\]
\end{proposition}

In order to prove the proposition, we need a result relating p-values and
quantiles. Its proof is given Appendix \ref{app:pvalue_vs_quantile}.

\begin{lemma}
\label{lem:pvalue_vs_quantile}
Let $Q$ be a distribution on $\R$, which is supported on finitely many
values. Then for any $\alpha\in[0,1]$ and $x\in\R$, 
\[ 
\Pp{Q}{X\geq x} > \alpha \iff x \leq \Quantile_{1-\alpha}(Q).
\]
\end{lemma}

\smallskip
\begin{proof}[Proof of Proposition \ref{prop:conformal_pvalue}]
By definition of the conformal prediction set, we have
\[
y\in\Ccal(X_{n+1}) \iff S^y_{n+1} \leq \Quantile_{1-\alpha}\bigg(\frac{1}{n+1}
\sum_{i=1}^{n+1}\delta_{S^y_i}\bigg). 
\]   
where recall $\frac{1}{n+1} \sum_{i=1}^{n+1}\delta_{S^y_i}$ is the empirical
distribution of the scores $S^y_1,\dots,S^y_{n+1}$. Applying Lemma
\ref{lem:pvalue_vs_quantile}, we have that for any $s\in\R$,  
\[
s \leq \Quantile_{1-\alpha}\bigg(\frac{1}{n+1}
\sum_{i=1}^{n+1} \delta_{S^y_i}\bigg) \iff \frac{1}{n+1}
\sum_{i=1}^{n+1} \One{S^y_i \geq s} > \alpha.
\]   
Taking $s=S^y_{n+1}$, and combining with the calculation above, we have
\[ 
y\in\Ccal(X_{n+1}) \iff \frac{1}{n+1} \sum_{i=1}^{n+1} \One{S^y_i \geq
  S^y_{n+1}} > \alpha,
\] 
which by definition of $p(y)$ in \eqref{eqn:conformal_pvalue}, completes the
proof. 
\end{proof}

The coverage result in Theorem \ref{thm:cp_background} can therefore be
translated into the following statement about the conformal p-value.

\begin{corollary}[\citealt{vovk2005algorithmic}]
\label{cor:conformal_pvalue}
If $Z_1,\dots,Z_{n+1}$ are exchangeable, and the score function is symmetric as
in \eqref{eqn:symmetry}, then the conformal p-value defined
in \eqref{eqn:conformal_pvalue} satisfies 
\[
\PP{p(Y_{n+1})\leq\alpha} \leq \alpha, \quad \text{for all $\alpha \in [0,1]$}. 
\]
\end{corollary}

One way to interpret Proposition \ref{prop:conformal_pvalue} and Corollary
\ref{cor:conformal_pvalue} is that the conformal prediction set is effectively
given by inverting a permutation test, leveraging the fact that all $(n+1)!$
orderings of the dataset $Z=(Z_1,\dots,Z_{n+1})$ are equally likely (due to
exchangeability). To make this more explicit, observe that the conformal p-value
in \eqref{eqn:conformal_pvalue} is equivalently
\[
p(y) = \frac{\sum_{\sigma\in\Scal_{n+1}} \One{\score(Z^y_{\sigma(n+1)},
    Z^y_\sigma) \geq \score(Z^y_{n+1}, Z^y)}}{(n+1)!},
\]
where recall $Z^y = (Z_1,\dots,Z_n,(X_{n+1},y))$, and we use \smash{$Z^y_\sigma
  = (Z^y_{\sigma(1)},\dots,Z^y_{\sigma(n+1)})$} for the result of permuting its 
entries under $\sigma$. If we think of $\teststat(Z^y) = \score(Z^y_{n+1}, Z^y)$
as a test statistic of the dataset $Z^y$, then the above reveals that $p(y)$ is
precisely a permutation p-value based on $\teststat$: 
\[
p(y) = \frac{\sum_{\sigma\in\Scal_{n+1}} \One{\teststat(Z^y_\sigma) \geq
    \teststat(Z^y)}}{(n+1)!}.
\]
The conformal prediction guarantee now becomes familiar: by the classical theory
of permutation testing, we immediately know that $p(Y_{n+1})$ is a valid p-value
since $Z^{Y_{n+1}} = Z$ has an exchangeable distribution.     

We note that the connections between conformal prediction and hypothesis testing
extend beyond standard CP to its many variants, such as WCP and NexCP. While
these and other extensions are typically directly via prediction sets, they can
also be represented through the language of conformal p-values. Our unified
framework (Sections \ref{sec:unified} and \ref{sec:unified_ext}) will also use
the language of hypothesis testing, to allow for a simple and clean
exposition. When we apply the unified framework to derive specific variants of
CP (Sections \ref{sec:cases_known} and \ref{sec:cases_new}), we will verify the
equivalence of its original formulation and p-value representation, in each
case.

\section{A unified framework}
\label{sec:unified}

We are now ready to build our unified framework for conformal prediction
methods. Recall that, in the standard exchangeable setting, the premise of
conformal prediction is that after conditioning on the unordered bag of data
$\Zbag$ we know the distribution of $Z$---it is simply uniform over all possible
permutations of the elements in the bag. For our unified framework, we will
generalize this idea in two ways: first, we will allow for settings that are
more general than exchangeability, and second, we will allow for conditioning on
more information than the bag of data $\Zbag$. To begin, we will need the
following two ingredients:
\begin{itemize}
\item \emph{Partial information} about the data, encoded by a random variable
  $U\in\Ucal$ on which we will condition to perform inference. We will assume
  that $U$ always contains sufficient information to reveal the unordered bag of
  data $\Zbag$, that is,  
  \[
  \Zbag = h(U) \;\, \text{almost surely},
  \]
  for some function $h$. In many cases, we will simply have $U=\Zbag$, but in
  others $U$ will contain additional information. Further, $U$ may be equal to
  a deterministic function of $Z$ (such as $U=\Zbag$), or may contain auxiliary 
  sources of randomness.      

\item A \emph{score function} $\score : \Zcal^{n+1}\times\Ucal\to\R$. As in
  standard CP, the value of $\score(Z,U)$ is intended to reflect the
  nonconformity of the last data point $Z_{n+1}$, relative to the observed
  data. That is, a large value of $\score(Z,U)$ indicates that $Z_{n+1}$ is
  likely an outlier.
\end{itemize}

We next need to determine the distribution of the statistic $\score(Z,U)$, in
order to compute a p-value. Essentially, we proceed by approximating the
conditional distribution of $Z \mid U$, and then computing a p-value for the
observed statistic $\score(Z,U)$ using this conditional distribution. This
requires a third ingredient: 
\begin{itemize}
\item An \emph{approximation} $Q_{Z\mid U}$ of the conditional distribution of 
  $Z\mid U$. Notice that the true conditional distribution of $Z\mid U$ is
  supported on the finite set $\{z\in\Zcal^{n+1} : \zbag = h(U)\}$ (as
  $\Zbag=h(U)$, almost surely). We assume $Q_{Z\mid U}$ also has support
  contained in this set. 
\end{itemize}

With these ingredients in place, we then define the p-value
\begin{equation}
\label{eqn:unified_pvalue}
p = p(Z,U) \quad \text{where} \quad p(z,u) = \Pp{Q_{Z \mid  U=u}}{\score(Z,u)
  \geq \score(z,u)}.  
\end{equation}

As explained in Section \ref{sec:conformal_pvalue}, conformal prediction methods 
can equivalently be written in terms of conformal p-values, thus predictive
coverage guarantees for $\PP{Y_{n+1}\in\Ccal(X_{n+1})}$ can be obtained by
bounding $\PP{p\leq \alpha}$ for the p-value constructed in
\eqref{eqn:unified_pvalue}.   

\subsection{Main theorem}

In order to examine the event $p\leq \alpha$, it will be useful to define the
quantity
\[
\thresh_\alpha(u) = \Quantile_{1-\alpha}(Q_{S\mid U=u}),
\]
where $Q_{S\mid U=u}$ denotes the conditional distribution of the score
$\score(Z,u)$ for $Z\sim Q_{Z\mid U=u}$. By Lemma \ref{lem:pvalue_vs_quantile},
since $Q_{S\mid U=u}$ is a distribution with finite support, it holds that
\begin{equation}
\label{eqn:pvalue_vs_quantile}
p(z,u) \leq \alpha \iff \score(z,u) > \thresh_\alpha(u).
\end{equation}
We now have the following guarantee for the p-value from the unified
framework. Here and henceforth, $\dtv(\cdot,\cdot)$ denotes the total  
variation (TV) distance between two distributions.

\begin{theorem}
\label{thm:unified}
Suppose $(Z,U) \sim P_{Z,U}$. Let $P_{S,T}$ be the induced joint distribution on
$(S,T)=(\score(Z,U),\thresh_\alpha(U))$. Then the p-value defined in
\eqref{eqn:unified_pvalue} satisfies   
\[
\PP{p\leq \alpha}\leq \alpha + \inf_{Q_U} \, \dtv(P_{S,T},Q_{S,T}),
\]
where the infimum is taken over all distributions $Q_U$ on $U$, and where
$Q_{S,T}$ denotes the joint distribution of
$(S,T)=(\score(Z,U),\thresh_\alpha(U))$ induced by drawing $(Z,U) \sim Q_{Z,U} =
Q_{Z\mid U}\times Q_U$.   
\end{theorem}

Before giving the proof, we pause to comment on some alternative versions of
this bound that are implied by the theorem. 

\begin{remark}
\label{rmk:information_monotonicity}
As $(S,T) = (\score(Z,U),\thresh_\alpha(U))$, we must have
$\dtv(P_{S,T},Q_{S,T})\leq \dtv(P_{Z,U},Q_{Z,U})$. This follows from a standard 
property of the total variation distance, which we will refer to as
\emph{information monotonicity}: 
\[
\text{$\dtv(P_{f(W)},Q_{f(W)})\leq \dtv(P_W,Q_W)$ for any distributions
  $P_W,Q_W$ and any function $f$.} 
\]
(Here $P_{f(W)}$ denotes the distribution of $f(W)$ for $W\sim P$,
and similarly for $Q_{f(W)}$.) Consequently, Theorem \ref{thm:unified} implies
a weaker bound, 
\[
\PP{p\leq \alpha}\leq \alpha + \inf_{Q_U} \, \dtv(P_{Z,U},Q_{Z,U}).
\]
The same argument holds more generally. For example, since $(S,T)$ is a function 
of $(S,U)$, we also have $\dtv(P_{S,T},Q_{S,T})\leq \dtv(P_{S,U},Q_{S,U})$. In
the applications examined later, we will use various versions of this type of
weaker bound, as convenient for each example.  
\end{remark}

\begin{remark}
\label{rmk:exact_validity}
Let $P_{Z\mid U}$ denote the conditional distribution associated with the joint
$P_{Z,U}$. If $Q_{Z\mid U} = P_{Z\mid U}$ holds almost surely, i.e., our choice
for the conditional distribution of $Z \mid U$ is exactly correct in
implementing the unified conformal method, then by taking $Q_U=P_U$, we obtain
$Q_{Z,U}=P_{Z,U}$ and therefore $\dtv(P_{Z,U},Q_{Z,U})=0$ which leads to
\[
\PP{p\leq \alpha}\leq \alpha.
\]
\end{remark}

\subsection{Proof of Theorem \ref{thm:unified}}

To build up toward the proof of the unified result, it is helpful to first
recall some standard properties of hypothesis testing. Suppose that we observe
data $Z$, and we have a prespecified test statistic---a function $\teststat$
which maps data $Z$ to a value $\teststat(Z)\in\R$ (with the convention that 
higher values indicate more evidence against the null). Fixing a null
hypothesis $H_0: Z\sim Q$ for a distribution $Q$, we may compute a
p-value $p(Z)$, where $p(z) = \Pp{Q}{\teststat(Z) \geq \teststat(z)}$ is the
probability of observing a statistic at least as large as $\phi(z)$ under the
null (for a one-sided test). By construction, the following holds: 
\begin{equation}
\label{eqn:pvalue_superuniform}
\text{if $Z\sim Q$ then} \;\, \PP{p(Z) \leq \alpha}\leq \alpha,
\end{equation}
for all $\alpha\in[0,1]$. But it can also be of interest to study the behavior
of the p-value $p$ under a different distribution, and by
\eqref{eqn:pvalue_superuniform} along with the definition of TV distance, we
have    
\begin{equation}
\label{eqn:pvalue_dtv_bound}
\text{if $Z\sim P$ then} \;\, \PP{p(Z)\leq \alpha}\leq \alpha + \dtv(P,Q).
\end{equation}
The above fact is often interpreted as a statement about power, i.e., if $P$ and
$Q$ are close in TV distance, then our test cannot have high power for rejecting
$Q$ in favor of the alternative $P$. However, we can also view
\eqref{eqn:pvalue_dtv_bound} as a statement about robustness to model
misspecification: if the null hypothesis of interest is $H_0: Z\sim P$,
and we only have access to an approximation $Q$ of this null distribution, then
the test may have inflated Type I error---but this inflation cannot be larger
than the total variation distance $\dtv(P,Q)$. 

With this intuition in place, we are now ready to prove the theorem.
At a high level, the proof can be viewed as a conditional version of the
standard bound \eqref{eqn:pvalue_dtv_bound} above. 

\begin{proof}[Proof of Theorem \ref{thm:unified}] 
By the standard fact \eqref{eqn:pvalue_superuniform} about p-values (applied
with $Q_{Z \mid U=u}$ in place of $Q$, and the score function $\score(\cdot,u)$
in place of the test statistic $\teststat$), we have  
\[
\Pp{Q_{Z \mid U=u}}{p(Z,u)\leq \alpha} \leq \alpha.
\]
As this holds for each $u\in\Ucal$, by averaging over $U\sim Q_U$, for any
marginal distribution $Q_U$, 
\[
\Pp{Q_{Z,U}}{p(Z,U)\leq \alpha} \leq \alpha. 
\]
Therefore, by definition of $Q_{S,T}$, along with the equivalence
\eqref{eqn:pvalue_vs_quantile} relating p-values to thresholds,
\[
\Pp{Q_{S,T}}{S>T} = \Pp{Q_{Z,U}}{\score(Z,U)> \thresh_\alpha(U)}
= \Pp{Q_{Z,U}}{p(Z,U)\leq \alpha} \leq \alpha.
\]
Finally, observe that by \eqref{eqn:pvalue_vs_quantile} once again,
\begin{align*}
\Pp{P_{Z,U}}{p(Z,U)\leq \alpha}
&=\Pp{P_{Z,U}}{\score(Z,U)> \thresh_\alpha(U)} \\
&=\Pp{P_{S,T}}{S> T} \\
&\leq \Pp{Q_{S,T}}{S> T} + \dtv(P_{S,T},Q_{S,T})\\
&\leq \alpha + \dtv(P_{S,T},Q_{S,T}),
\end{align*}
where the next-to-last step holds by definition of total variation distance.
\end{proof}

\section{Special cases: known results}
\label{sec:cases_known}

In this section, we examine a range of special cases in which the unified
framework is able to reproduce known results in the literature. In
each subsection below, we review the setting underlying a particular conformal
method, the definition of the method itself, the associated theory, and then
demonstrate how this can be viewed from the lens of the unified framework. 

\subsection{Standard conformal prediction}

Standard conformal prediction (CP) \citep{vovk2005algorithmic} works in the
setting where the data $Z_1,\dots,Z_{n+1}$ are exchangeable. This method was
already described above in Section \ref{sec:cp_background}, but we briefly
review it again here for completeness before describing its reformulation under  
the unified framework.

\subsubsection{Method and theory}

Fix any score function of the form $\score((x,y),\zbag)$, which assigns a value
to a single data point (its first argument), based on a bag of data points (its
second argument). For arbitrary $y \in \Ycal$, we define the augmented dataset
$Z^y = (Z_1,\dots,Z_n,(X_{n+1},y))$, and scores   
\[
S^y_i = \score(Z^y_i, \lbag Z^y \rbag), \quad i\in[n+1].
\] 
We define the CP set at coverage level $1-\alpha$ by 
\begin{equation}
\label{eqn:cp}
\Ccal(X_{n+1}) = \bigg\{ y \in \Ycal: S^y_{n+1} \leq \Quantile_{1-\alpha}
\bigg(\frac{1}{n+1} \sum_{i=1}^{n+1} \delta_{S^y_i} \bigg) \bigg\}. 
\end{equation}
Recalling Theorem \ref{thm:cp_background}, if the data points
$Z_1,\dots,Z_{n+1}$ are exchangeable, then CP provides a marginal coverage 
guarantee $\PP{Y_{n+1}\in\Ccal(X_{n+1})} \geq 1-\alpha$.   

\subsubsection{View from the unified framework}

We now describe how CP fits into the unified conformal framework.

\paragraph{Choices of $U$ and $Q_{Z \mid U}$.}

We define the partial information as $U = \Zbag$. In standard CP, given the bag 
$u=\zbag$, the score assigned to $z$ depends only on $z_{n+1}$. Overloading
notation, we can thus write the score function as
\[
\score(z,u) = \score(z_{n+1},\zbag).
\]
Next we take the choice of the conditional distribution $Q_{Z\mid U}$ to be
\[
Q_{Z \mid U=\zbag} = \frac{1}{(n+1)!} \sum_{\sigma \in \Scal_{n+1}}
\delta_{z_\sigma}.
\]
This is the uniform distribution over all permutations of $z$, i.e., all vectors
consistent with the observed bag $\zbag$ of data points.

\paragraph{P-value.}

Under these choices, the p-value in \eqref{eqn:unified_pvalue} is  
\begin{align*}
p &= \frac{1}{(n+1)!} \sum_{\sigma \in \Scal_{n+1}}
  \One{\score(Z_{\sigma(n+1)},\Zbag) \geq \score (Z_{n+1},\Zbag)} \\ 
&= \frac{1}{n+1} \sum_{i=1}^{n+1} \One{\score(Z_i,\Zbag) \geq \score
  (Z_{n+1},\Zbag)}, 
\end{align*}
where the second line follows from the fact that, for each $i$, there are $n!$
many permutations $\sigma\in\Scal_{n+1}$ with $\sigma(n+1)=i$. As we can see,
this is the conformal p-value $p(Y_{n+1})$ defined in
\eqref{eqn:conformal_pvalue}. By Proposition \ref{prop:conformal_pvalue}, the 
event $p>\alpha$ is equivalent to $Y_{n+1}\in\Ccal(X_{n+1})$, or, in other
words, bounding $\PP{p\leq \alpha}$ (as we will do next via the unified theory)
is equivalent to providing a predictive coverage guarantee for CP.

\paragraph{Validity.}

We present an alternative derivation of the standard conformal prediction theory
in Theorem \ref{thm:cp_background}, based on the unified result in Theorem
\ref{thm:unified}. 

\begin{proof}[Proof of Theorem \ref{thm:cp_background} via the unified
  framework]  
Recall that $U=\Zbag$. As have assumed the data is exchangeable, the true
conditional distribution is given by  
\[
P_{Z\mid U=\zbag} = \frac{1}{(n+1)!}\sum_{\sigma\in\Scal_{n+1}}
\delta_{z_\sigma},
\]
i.e., after observing the unordered bag of data $\Zbag = \zbag$, each
permutation of $z$ is equally likely. Since $P_{Z\mid U}=Q_{Z\mid U}$, this
proves (recalling Remark \ref{rmk:exact_validity}) that $\PP{p\leq \alpha}\leq
\alpha$.  
\end{proof}

This template---providing a predictive coverage guarantee by bounding the Type I 
error of the associated p-value---is used in all variants of conformal
prediction in this section. 

\subsection{Split conformal prediction}

Split conformal prediction is a widely-used variant of standard conformal
prediction where a data split is used in order to facilitate computation. Split
CP was proposed by \citet{papadopoulos2002inductive} and studied in
\citet{lei2018distribution}, among many others. The problem setting that we
consider here is precisely as in the last subsection: we assume exchangeable 
$Z_1,\dots,Z_{n+1}$. (We note that split variants also exist for all conformal
methods that will be discussed in the coming subsections, but for simplicity, we
only study it as a variant of standard CP.)  

\subsubsection{Method and theory}

We first partition the dataset $(Z_1,\dots,Z_{n+1})$ into two parts, written as
$Z_{(0)}=(Z_1,\dots,Z_{n_0})$, and $Z_{(1)}=(Z_{n_0+1},\dots,Z_{n+1})$. 
Similarly, for any $z\in\Zcal^{n+1}$, we will write
$z_{(0)}=(z_1,\dots,z_{n_0})$, and $z_{(1)}=(z_{n_0+1},\dots,z_{n+1})$. The idea 
of split conformal prediction to train a model on $Z_{(0)}$, then compute scores
for all remaining data points in order to build a prediction set for the test
point. Concretely, we will work with a score function of the form
\[
\score((x,y), z_{(0)}).
\]
This accommodates choices of the form $\score((x,y), z_{(0)}) =
\ell(y,\hat{f}(x))$, for $\hat{f} = \alg(z_{(0)})$, as in
\eqref{eqn:loss_score}. Notice that $\alg$ here is not required to treat its
input data points symmetrically. The split CP set at coverage level $1-\alpha$
is then given by 
\begin{equation}
\label{eqn:scp}
\Ccal(X_{n+1}) = \bigg\{ y \in \Ycal: \score((X_{n+1},y),Z_{(0)}) \leq
\Quantile_{(1-\alpha)(1+\frac{1}{n_1})} \bigg(\frac{1}{n_1} \sum_{i=n_0+1}^n
\delta_{\score(Z_i,Z_{(0)})}\bigg) \bigg\}, 
\end{equation}
where we write $n_1 = n-n_0$. Split CP is known to have training-conditional
coverage under exchangeability.  

\begin{theorem}[\citealt{vovk2005algorithmic}]
\label{thm:scp}
If $Z_1,\dots,Z_{n+1}$ are exchangeable, then the split CP set given in
\eqref{eqn:scp} satisfies 
\[
\PPst{Y_{n+1} \in \Ccal(X_{n+1})}{Z_{(0)}} \geq 1-\alpha,
\]
almost surely. In particular, this implies $\PP{Y_{n+1} \in \Ccal(X_{n+1})} \geq
1-\alpha$.   
\end{theorem}

\subsubsection{View from the unified framework}

We now describe how split CP fits into the unified conformal framework.

\paragraph{Choices of $U$ and $Q_{Z \mid U}$.} 

Unlike in standard CP where we simply take $U=\Zbag$, here we will use a
different choice of partial information $U$: for split CP, we define $U =
(\Zbag, Z_{(0)})$.\footnote{Note that an equivalent choice would be $U =
  (\lbag Z_{(1)} \rbag, Z_{(0)})$, since this would contain the same information 
  about the data $Z$.} Overloading notation again, we write the score for any 
$u=(\zbag,z_{(0)})$ as   
\[
\score(z, u) = \score(z_{n+1}, z_{(0)}).
\]
That is, the score assigned to $z$ depends only on comparing $z_{n+1}$ (the last
observation) to the first portion $z_{(0)}$ of the training data. Next we define
\[
Q_{Z \mid U=(\zbag,z_{(0)})} = \frac{1}{(n_1+1)!} \sum_{\sigma \in
  \Scal_{n_1+1}} \delta_{(z_{(0)}, [z_{(1)}]_\sigma)}.
\]
This is the empirical distribution on all permutations of $z$ consistent
with $z_{(0)}$---that is, this distribution places equal weight on each vector  
of the form $(z_{(0)}, [z_{(1)}]_\sigma)$, which preserves the first $n_0$ entries
of $z$ (i.e., $z_{(0)}$) in their original order, and allows the remaining
$n_1+1$ entries (i.e., $z_{(1)}$) to be permuted arbitrarily.   

\paragraph{P-value.}

Under these choices, the p-value in \eqref{eqn:unified_pvalue} is 
\begin{align*}
p &= \frac{1}{(n_1+1)!} \sum_{\sigma \in \Scal_{n_1+1}}
  \One{\score([Z_{(1)}]_{\sigma(n_1+1)},Z_{(0)})\geq \score(Z_{n+1},Z_{(0)})} \\   
&= \frac{1}{n_1+1} \sum_{i=n_0+1}^{n+1} \One{\score(Z_i,Z_{(0)})\geq
  \score(Z_{n+1},Z_{(0)})},  
\end{align*}
where the second line follows similarly to the simplification used in standard
conformal. 

To see that this corresponds to split CP, recall that by Lemma
\ref{lem:pvalue_vs_quantile},   
\[
p > \alpha \iff \score(Z_{n+1},Z_{(0)}) \leq \Quantile_{1-\alpha}\bigg(
\frac{1}{n_1+1} \sum_{i=n_0+1}^{n+1} \delta_{\score(Z_i,Z_{(0)})} \bigg). 
\]
Next we will need an elementary fact about quantiles: for any values
$v_1,\dots,v_{m+1} \in \R$,  
\[
v_{m+1} \leq \Quantile_{1-\alpha}(v_1,\dots,v_{m+1}) \iff v_{m+1} \leq
\Quantile_{(1-\alpha)(1+1/m)}(v_1,\dots,v_m). 
\]
This implies that
\[
p > \alpha \iff \score(Z_{n+1},Z_{(0)}) \leq
\Quantile_{(1-\alpha)(1+\frac{1}{n_1})}\bigg( \frac{1}{n_1}\sum_{i=n_0+1}^n
\delta_{\score(Z_i,Z_{(0)})} \bigg). 
\]
The right-hand side above can be directly seen to be equivalent to the event
$Y_{n+1} \in \Ccal(X_{n+1})$ for the split CP set in \eqref{eqn:scp}. 

\paragraph{Validity.}

We prove Theorem \ref{thm:scp} using the unified result in Theorem
\ref{thm:unified}. 

\begin{proof}[Proof of Theorem \ref{thm:scp} via the unified framework] 
Since $Z=(Z_1,\dots,Z_{n+1})$ is exchangeable, it also holds that $Z_{(1)}\mid
Z_{(0)}=z_{(0)}$ is exchangeable, for almost every $z_{(0)}$. Fix any $z_{(0)}$
for which this conditional exchangeability holds. Let $P_Z$ denote the
distribution of $Z\mid Z_{(0)}=z_{(0)}$, and recalling that $U=(\Zbag,Z_{(0)})$, 
write $P_{Z,U}$ as the induced joint distribution on $(Z,U)$. Then we can
calculate the true conditional distribution as
\[
P_{Z \mid U=(\zbag,z_{(0)})} = \frac{1}{(n_1+1)!} \sum_{\sigma \in
  \Scal_{n_1+1}} \delta_{(z_{(0)}, [z_{(1)}]_\sigma)}, 
\]
because the distribution of $Z_{(1)}$ conditional on $Z_{(0)}=z_{(0)}$ is 
exchangeable. We therefore see that $P_{Z\mid U}=Q_{Z\mid U}$, which proves 
(recalling Remark \ref{rmk:exact_validity}) that $\PP{p\leq \alpha}\leq
\alpha$.    
\end{proof}

\subsection{Weighted conformal prediction}
\label{sec:wcp}

Weighted conformal prediction (WCP) works in a problem setting where 
$Z_1,\dots,Z_n$ are i.i.d.\ from $\Ptr$, whereas the test point $Z_{n+1}$ is
drawn independently from $\Pte$.\footnote{The results for WCP hold under a
  weaker assumption than what is stated here, called \emph{weighted
    exchangeability}, as defined in \citet{tibshirani2019conformal}. For
  simplicity, we work with independent data in our treatment here, but we note
  that the unified framework can also encompass the weighted exchangeable
  setting.} While $\Ptr$ and $\Pte$ are generally unknown, we additionally
assume that we have knowledge of the \emph{distribution shift} relating the
likelihood of $\Pte$ to $\Ptr$,  
\begin{equation}
\label{eqn:wcp_wstar}
w^*(x,y) = \frac{\d\Pte}{\d\Ptr}(x,y).
\end{equation}
As our knowledge of this shift might be only approximate, we will work with a
user-specified weight function $w$ that approximates $w^*$. WCP was introduced 
by \citet{tibshirani2019conformal}, who focused on covariate shift (where
$w^*(x,y)$ depends only on $x$). It has also been studied in other settings,
such as label shift \citep{podkopaev2021distribution}, causal inference
\citep{lei2021conformal}, and survival analysis
\citep{candes2023conformalized}. 

\subsubsection{Method and theory}

Fix any score function of the form $\score((x,y),\zbag)$ (as in standard CP, the
score here treats the dataset $z$ symmetrically). For arbitrary $y \in \Ycal$,
define as usual $Z^y = (Z_1,\dots,Z_n,(X_{n+1},y))$, and scores  
\[
S^y_i = \score(Z^y_i, \lbag Z^y \rbag), \quad i\in[n+1].
\]
Given the user-specified weight function $w:\Zcal\to(0,\infty)$, we define the
WCP set at coverage level $1-\alpha$ by   
\begin{equation}
\label{eqn:wcp}
\Ccal(X_{n+1}) = \bigg\{ y \in \Ycal: S^y_{n+1} \leq \Quantile_{1-\alpha}
\bigg(\frac{\sum_{i=1}^n w(Z_i) \cdot \delta_{S^y_i} + w(X_{n+1},y) \cdot
  \delta_{S^y_{n+1}}}{\sum_{i=1}^n w(Z_i) + w(X_{n+1},y)} \bigg) \bigg\}. 
\end{equation}
WCP is known to have exact coverage if $w = w^*$, and approximate coverage if  
$w \approx w^*$. 

\begin{theorem}
\label{thm:wcp}
For independent $Z_1,\dots,Z_n \sim \Ptr$ and $Z_{n+1} \sim \Pte$, where $w^* =
\d\Pte/\d\Ptr$, the WCP set in \eqref{eqn:wcp} satisfies the following: 

\begin{itemize}[align=left,labelwidth={1.25em}]
\item[(a) \textnormal{\citep{tibshirani2019conformal}}.] If $w = w^*$, then
  $\PP{Y_{n+1} \in \Ccal(X_{n+1})} \geq 1-\alpha$.  

\item[(b) \textnormal{\citep{lei2021conformal}}.] Assuming $\Ep{\Ptr}{w(X,Y)} <
  \infty$, and defining a normalized version of $w$ by $\bar{w}(x,y) = w(x,y) /
  \Ep{\Ptr}{w(X,Y)}$, we have  
  \[
    \PP{Y_{n+1} \in \Ccal(X_{n+1})} \geq 1-\alpha - \frac{1}{2}
    \Ep{\Ptr}{|\bar{w}(X,Y) - w^*(X,Y)|}. 
  \]
\end{itemize}
\end{theorem}

\subsubsection{View from the unified framework}

We describe how WCP fits into the unified conformal framework.

\paragraph{Choices of $U$ and $Q_{Z \mid U}$.}

Define $U = \Zbag$, and overloading notation again, write the score
function for $u=\zbag$ as     
\[
\score(z,u) = \score(z_{n+1},\zbag).
\]
Next define 
\[
Q_{Z \mid U=\zbag} = \frac{\sum_{\sigma \in \Scal_{n+1}} w(z_{\sigma(n+1)})
  \cdot \delta_{z_\sigma}}{\sum_{\sigma \in \Scal_{n+1}} w(z_{\sigma(n+1)})}, 
\]
which is the weighted empirical distribution that places weight proportional to
$w(z_{\sigma(n+1)})$ on each permutation $z_\sigma$.  

\paragraph{P-value.}

Under these choices, the p-value in \eqref{eqn:unified_pvalue} is 
\begin{align*}
p &= \frac{\sum_{\sigma \in \Scal_{n+1}} w(Z_{\sigma(n+1)}) \cdot
  \One{\score(Z_{\sigma(n+1)}, \Zbag) \geq \score(Z_{n+1},
  \Zbag)}}{\sum_{\sigma \in \Scal_{n+1}} w(Z_{\sigma(n+1)})} \\ 
&= \frac{\sum_{i=1}^{n+1} w(Z_i) \cdot \One{\score(Z_i, \Zbag) \geq 
  \score(Z_{n+1}, \Zbag)}}{\sum_{i=1}^{n+1} w(Z_i)},
\end{align*}
where the second line follows from the fact that, for each $i$, there are $n!$
many permutations $\sigma\in\Scal_{n+1}$ with $\sigma(n+1)=i$. To see that this
corresponds to WCP, note that by Lemma \ref{lem:pvalue_vs_quantile},    
\[
p > \alpha \iff \score(Z_{n+1}, \Zbag) \leq \Quantile_{1-\alpha} \bigg(
\frac{\sum_{i=1}^{n+1} w(Z_i) \cdot \delta_{\score(Z_i,
    \Zbag)}}{\sum_{i=1}^{n+1} w(Z_i)} \bigg). 
\]
The right-hand side above can be directly seen to be equivalent to the event
$Y_{n+1} \in \Ccal(X_{n+1})$, for the WCP set in \eqref{eqn:wcp}. 

\paragraph{Validity.}

We prove Theorem \ref{thm:wcp} using the unified result in Theorem
\ref{thm:unified}. 

\begin{proof}[Proof of Theorem \ref{thm:wcp} via the unified framework]
Recalling $U = \Zbag$, define $Q_U$ to be the distribution of $\Zbag$ when $Z
\sim H$, where  
\[
\d{H}(z) = \frac{\sum_{i=1}^{n+1} \bar{w}(z_i)}{n+1} \cdot \d\Ptr^{n+1}(z).
\]
For intuition, note that we can interpret this as a mixture: we sample an index
$i\in[n+1]$ uniformly at random, then sample $Z_i$ from a distribution $\Ptr
\circ \bar{w}$, defined by $\d(\Ptr \circ \bar{w})(x,y) = w(x,y) \cdot
\d\Ptr(x,y)$, whereas the remaining $n$ data points are drawn as \smash{$Z_j
  \iidsim \Ptr$}, for $j \not= i$. With $Q_{Z,U} = Q_{Z\mid U}\times Q_U$, as
usual, we examine the corresponding marginal $Q_Z$ on $Z$: by definition of
$Q_U$ and $Q_{Z\mid U}$, we have     
\allowdisplaybreaks
\begin{align*}
Q_Z(A) &= \Ep{Z \sim H}{\frac{\sum_{\sigma \in \Scal_{n+1}} w(Z_{\sigma(n+1)}) 
  \cdot \One{Z_\sigma \in A}}{\sum_{\sigma \in \Scal_{n+1}} w(Z_{\sigma(n+1)})}} 
  \\  
&= \Ep{Z \sim \Ptr^{n+1}}{\frac{\sum_{i=1}^{n+1} \bar{w}(Z_i)}{n+1} \cdot
  \frac{\sum_{\sigma \in \Scal_{n+1}} w(Z_{\sigma(n+1)}) \cdot
  \One{Z_\sigma \in A}}{\sum_{\sigma \in \Scal_{n+1}} w(Z_{\sigma(n+1)})}} \\  
&= \Ep{Z \sim \Ptr^{n+1}}{\frac{\sum_{\sigma \in \Scal_{n+1}} 
  \bar{w}(Z_{\sigma(n+1)})}{(n+1)!} \cdot \frac{\sum_{\sigma \in \Scal_{n+1}}  
  \bar{w}(Z_{\sigma(n+1)}) \cdot \One{Z_\sigma \in A}}
  {\sum_{\sigma \in \Scal_{n+1}} \bar{w}(Z_{\sigma(n+1)})}} \\
&= \Ep{Z \sim \Ptr^{n+1}}{\frac{1}{(n+1)!}\sum_{\sigma \in \Scal_{n+1}}  
  \bar{w}(Z_{\sigma(n+1)}) \cdot \One{Z_\sigma \in A}} \\
&= \Ep{Z \sim \Ptr^{n+1}}{\bar{w}(Z_{n+1})\cdot \One{Z \in A}} \\ 
&= \Ep{Z \sim \Ptr^n \times (\Ptr \circ \bar{w})}{\One{Z \in A}},
\end{align*}
where the next-to-last line holds by exchangeability of $\Ptr^{n+1}$.  

Meanwhile, the true distribution on $Z$ is $P_Z = \Ptr^n \times \Pte = \Ptr^n 
\times (\Ptr \circ w^*)$. Theorem \ref{thm:unified} thus gives the Type I error
bound,
\[
\PP{p\leq \alpha} \leq \alpha + \dtv(P_{S,T},Q_{S,T}) \leq \alpha +
\dtv(P_Z,Q_Z), 
\]
where (recalling information monotonicity, in Remark
\ref{rmk:information_monotonicity}) the last step uses the fact that $(S,T)$ is
a function of $Z$.   

To complete the proof, we only need to bound $\dtv(P_Z,Q_Z)$. We have 
\begin{align*}
\dtv(P_Z,Q_Z)
&=\dtv(\Ptr^n\times (\Ptr\circ w^*), \Ptr^n\times (\Ptr\circ \bar{w})) \\
&=\dtv(\Ptr\circ w^*, F\circ \bar{w} )\\
&=\frac{1}{2}\int_\Zcal \left|w^*(x,y) - \bar{w}(x,y)\right| \;\d\Ptr(x,y) \\ 
&=\frac{1}{2}\Ep{\Ptr}{|\bar{w}(X,Y)-w^*(X,Y)|},
\end{align*}
where the next-to-last line applies the well-known $L^1$ representation for TV
distance. This recovers the result in part (b). In the special case where
$w=w^*$ (and so $\bar{w}=w^*$), we have $\Ep{\Ptr}{|\bar{w}(X,Y)-w^*(X,Y)|}=0$,
which recovers the result in part (a). 
\end{proof}

\subsection{Nonexchangeable conformal prediction}
\label{sec:nexcp}

Nonexchangeable conformal prediction (NexCP) works in a setting where
$Z_1,\dots,Z_{n+1}$ are drawn from an arbitrary joint distribution, but there is
some underlying structure (e.g., the samples are indexed by time or space)
allowing us to posit a guess as to which of $Z_1,\dots,Z_n$ are ``closer'' in
distribution to $Z_{n+1}$. NexCP was introduced by \citet{barber2023conformal}. 

\subsubsection{Method and theory}

Fix any score function $\score((x,y),z)$, which assigns a value to a data point
(its first argument) based on an \emph{ordered} dataset (its second
argument). Notice that this score function is not required to be symmetric in
$z$, and accommodates choices of the form $\score((x,y), z) =
\ell(y,\hat{f}(x))$, for a fitted model $\hat{f} = \alg(z)$, as in
\eqref{eqn:loss_score}, where the algorithm $\alg$ does not need to treat its
input data points symmetrically.  

Given user-specified weights $w_1,\dots,w_{n+1} \geq 0$, with $\sum_{k=1}^{n+1}
w_k = 1$, we sample an index 
\begin{equation}
\label{eqn:nexcp_k}
K \sim \sum_{k=1}^{n+1} w_k \cdot \delta_k.
\end{equation}
We then define, for arbitrary $y \in \Ycal$, the augmented dataset $Z^y = 
(Z_1,\dots,Z_n,(X_{n+1},y))$ as usual, and scores
\[
S^{y,K}_i = \score(Z^y_i, (Z^y)^K), \quad i\in[n+1],
\]
where for any $z\in\R^{n+1}$ and $k\in[n+1]$, we use $z^k$ to denote the 
vector obtained by swapping the entries in positions $k$ and $n+1$ of $z$, i.e.,  
\[
z^k = (z_1,\dots,z_{k-1},z_{n+1},z_{k+1},\dots,z_n,z_k)\in\Zcal^{n+1}
\]
(and, to handle the case $k=n+1$, we simply define $z^{n+1}=z$). We now
define the NexCP set at coverage level $1-\alpha$ by     
\begin{equation}
\label{eqn:nexcp}
\Ccal(X_{n+1}) = \bigg\{ y \in \Ycal: S^{y,K}_{n+1} \leq \Quantile_{1-\alpha}
\bigg(\sum_{i=1}^{n+1} w_i \cdot \delta_{S^{y,K}_i} \bigg) \bigg\}. 
\end{equation}
NexCP is known to provide a coverage guarantee for arbitrary joint
distributions.   

\begin{theorem}[\citealt{barber2023conformal}]
\label{thm:nexcp}
For $Z_1,\dots,Z_{n+1}$ of arbitrary joint distribution, if $w_{n+1}\geq w_i$
for all $i\in[n+1]$, then the NexCP set in \eqref{eqn:nexcp} satisfies 
\[
\PP{Y_{n+1}\in\Ccal(X_{n+1})} \geq 1 - \alpha - \sum_{k=1}^{n+1} w_k \cdot
\dtv(g(Z),g(Z^k)), 
\]
where the function $g:\Zcal^{n+1}\to\R^{n+1}$ is defined as     
\[
g(z) = \big( \score(z_1,z),\dots,\score(z_{n+1},z) \big).
\]
In particular, by information monotonicity, this implies that
$\PP{Y_{n+1}\in\Ccal(X_{n+1})} \geq 1 - \alpha - \sum_{k=1}^{n+1} w_k \cdot
\dtv(Z,Z^k)$. 
\end{theorem}

This result can be interpreted as a guaranteeing that coverage holds at
approximately the desired level $1-\alpha$, as long as the dataset $Z$---or 
rather, its corresponding vector of scores, $g(Z)$---is approximately
exchangeable.  

\subsubsection{View from the unified framework}

We describe how NexCP fits into the unified conformal framework. 

\paragraph{Choices of $U$ and $Q_{Z\mid U}$.}

Set $U=Z^K$, where $K$ is drawn as in \eqref{eqn:nexcp_k}. Notice that the
choice of the partial information $U$ in this setting contains additional
information beyond the bag of data. Overloading notation, we write the score
function as  
\[
\score(z,u) = \score(z_{n+1},u).
\]
Next define
\[
Q_{Z \mid U=u} = \sum_{k=1}^{n+1} w_k \cdot \delta_{u^k},
\]
placing weight $w_k$ on each swapped vector $u^k$. 

\paragraph{P-value.}

Under these choices, the p-value in \eqref{eqn:unified_pvalue} is
\[
p = \sum_{k=1}^{n+1} w_k \cdot \One{\score((Z^K)_k,Z^K) \geq
  \score(Z_{n+1},Z^K)}.
\]
On the event $K\in[n]$, then we calculate
\[
p = \sum_{k\in[n]\backslash\{K\}} w_k \cdot \One{\score(Z_k,Z^K) \geq
  \score(Z_{n+1},Z^K)} + w_K + w_{n+1} \cdot \One{\score(Z_K,Z^K) \geq
  \score(Z_{n+1},Z^K)}, 
\] 
using the fact that $(Z^K)_k = Z_k$ for all $k \in[n]\backslash\{K\}$, whereas
$(Z^K)_K = Z_{n+1}$ and $(Z^K)_{n+1} = Z_K$. Since $w_{n+1}\geq w_K$ by
assumption, we therefore have 
\begin{align*}
p &\leq \sum_{k\in[n]\backslash\{K\}} w_k \cdot \One{\score(Z_k,Z^K) \geq
  \score(Z_{n+1},Z^K)} + w_{n+1} + w_K \cdot \One{\score(Z_K,Z^K) \geq
  \score(Z_{n+1},Z^K)} \\
&= \sum_{k=1}^{n+1} w_k \cdot \One{\score(Z_k,Z^K) \geq \score(Z_{n+1},Z^K)}
=: p^*.
\end{align*} 
If instead $K=n+1$, then by definition we simply have $p=p^*$. Based on Lemma
\ref{lem:pvalue_vs_quantile}, we can see that $Y_{n+1} \in \Ccal(X_{n+1})$ for
NexCP \eqref{eqn:nexcp} is equivalent to $p^* > \alpha$. Since we have shown
that $p\leq p^*$, note that $\PP{p^* \leq\alpha}\leq \PP{p\leq \alpha}$ and
thus a Type I error bound on $p$ will translate into one for the NexCP p-value
$p^*$.  

\paragraph{Validity.}

We prove Theorem \ref{thm:nexcp} using the unified result in Theorem
\ref{thm:unified}. 

\begin{proof}[Proof of Theorem \ref{thm:nexcp} via the unified framework]
By definition of $Q_{Z \mid U}$, observe that
\[
\thresh_\alpha(U) = \Quantile_{1-\alpha} \bigg(\sum_{k=1}^{n+1} w_k \cdot 
\delta_{\score((U^k)_{n+1},U)}\bigg) = \Quantile_{1-\alpha}
\bigg(\sum_{k=1}^{n+1} w_k \cdot \delta_{\score(U_k,U)}\bigg),
\]
and thus, by definition of $g$,
\[
\thresh_\alpha(U) = \Quantile_{1-\alpha}\bigg(\sum_{k=1}^{n+1} w_k \cdot
\delta_{g(U)_k}\bigg).    
\]
Similarly, for any $k\in[n+1]$,
\[
\score(U^k,U) = \score((U^k)_{n+1},U) = \score(U_k,U) = g(U)_k.
\]
For convenience, define 
\[
h_k(v) = \bigg( v_k, \, \Quantile_{1-\alpha}\bigg(\sum_{k'=1}^{n+1} w_{k'} 
  \cdot \delta_{v_{k'}}\bigg) \bigg).
\]
By construction, then, for any $k\in[n+1]$, we have
\[
(\score(U^k,U),\thresh_\alpha(U))= h_k(g(U)).
\]
Therefore, for $U=Z^K$, we have
\[
(S,T) = (\score(Z,U),\thresh_\alpha(U)) = (\score(U^K,U),\thresh_\alpha(U)) =
h_K(g(U)) = h_K(g(Z^K)).
\]
Since by construction we sample $K\sim \sum_{k=1}^{n+1} w_k \cdot \delta_k$  
independently of $Z$, note that under the joint distribution $(Z,U)\sim
P_{Z,U}$, we have   
\[
P_{S,T} = \sum_{k=1}^{n+1} w_k \cdot P_{h_k(g(Z^k))},
\]
where $P_{h_k(g(Z^k))}$ denotes the distribution of $h_k(g(Z^k))$ induced by
$Z\sim P_Z$.   

Next, we choose to define $Q_U=P_Z$, and now we need to calculate the 
distribution $Q_{S,T}$. Since $Q_{Z\mid U} = \sum_{k=1}^{n+1} w_k \cdot
\delta_{U^k}$, we can therefore calculate that under $Q_{Z\mid U}$, 
\[
(S,T) = (\score(Z,U), \thresh_\alpha(U)) \sim \sum_{k=1}^{n+1} w_k \cdot
\delta_{(\score(U^k,U),\thresh_\alpha(U))} = \sum_{k=1}^{n+1} w_k \cdot 
\delta_{h_k(g(U))}. 
\]
In other words, we have shown 
\[
Q_{S,T} = \sum_{k=1}^{n+1} w_k \cdot Q_{h_k(g(U))},
\]
where $Q_{h_k(g(U))}$ denotes the distribution of $h_k(g(U))$ induced by $U \sim
Q_U$.  

We therefore calculate
\begin{align*}
\dtv(P_{S,T},Q_{S,T}) 
&= \dtv\bigg(\sum_{k=1}^{n+1} w_k \cdot P_{h_k(g(Z^k))}, \sum_{k=1}^{n+1} w_k  
  \cdot Q_{h_k(g(U))}\bigg) \\ 
&\leq \sum_{k=1}^{n+1}w_k\cdot \dtv\Big(P_{h_k(g(Z^k))},Q_{h_k(g(U))}\Big) \\ 
&\leq \sum_{k=1}^{n+1}w_k\cdot \dtv\Big(P_{g(Z^k)},Q_{g(U)}\Big) \\
&= \sum_{k=1}^{n+1}w_k\cdot \dtv\Big(P_{g(Z^k)},P_{g(Z)}\Big),
\end{align*}
where the third step uses the information monotonicity property of total
variation distance (recall Remark \ref{rmk:information_monotonicity}), while the 
last step uses $Q_U=P_Z$. Therefore, 
\[
\PP{p\leq \alpha}\leq \alpha + \dtv(P_{S,T},Q_{S,T})\leq \alpha +
\sum_{k=1}^{n+1}w_k \cdot \dtv(P_{g(Z^k)},P_{g(Z)}),
\]
which completes the proof.
\end{proof}

\subsection{Randomly-localized conformal prediction}
\label{sec:rlcp}

Localized conformal methods work in a setting where $Z_1,\dots,Z_{n+1}$ are
exchangeable but we seek to go beyond the marginal guarantee
$\PP{Y_{n+1}\in\Ccal(X_{n+1})} \geq 1-\alpha$ offered by standard CP. A stronger
coverage guarantee that holds conditionally on the value of the test feature, as
in $\PPst{Y_{n+1}\in\Ccal(X_{n+1})}{X_{n+1}=x} \geq 1-\alpha$, would be a highly
desirable (and aspirational) goal. Unfortunately, when the covariate $X$ is
continuously distributed, well-known negative results by
\citet{vovk2012conditional, lei2014distribution} show this is not possible to
achieve in a distribution-free manner, excepting trivial (infinitely wide)
prediction intervals.

Researchers have therefore turned to developing procedures with
\emph{approximate} conditional coverage over small neighborhoods in the feature
space $\Xcal$. \citet{guan2023localized} introduced localized-conformal
prediction (LCP), which computes a prediction set by assigning higher weight to
data points that lie closer to the test point $X_{n+1}$. LCP gives marginal
coverage, and leads to approximate conditional coverage in practice, but does
not have accompanying finite-sample theory. To address this,
\citet{hore2023conformal} derived a variant called randomly-localized conformal
prediction (RLCP), which we study in this subsection.

\subsubsection{Method and theory}

Fix any score function of the form $\score((x,y),\zbag)$. For arbitrary $y \in
\Ycal$, we now define as usual $Z^y = (Z_1,\dots,Z_n,(X_{n+1},y))$, and scores     
\[
S^y_i = \score(Z^y_i, \lbag Z^y\rbag), \quad i\in[n+1],
\]
The user specifies a ``localizing'' kernel $H:\Xcal\times\Xcal\to\R_+$, e.g., in
the case $\Xcal=\R^d$, we can use the Gaussian kernel $H(x,x') =
e^{-\|x-x'\|^2_2/2\sigma^2} / (2\pi\sigma^2)^{d/2}$. We assume that, for any
$x\in\Xcal$, $H(x,\cdot)$ is a density (relative to some base measure), with
$H(x,x)>0$. We then sample
\[
\tilde{X}_{n+1}\mid X_{n+1} \sim H(X_{n+1},\cdot),
\]
which is then used to define weights in the RLCP set at coverage level $1-\alpha$,
defined by\footnote{\citet{hore2023conformal} describe a split implementation of
  RLCP, whereas what we describe here is the corresponding full-data version.} 
\begin{equation}
\label{eqn:rlcp}
\Ccal(X_{n+1}) = \bigg\{ y \in \Ycal: S^y_{n+1} \leq \Quantile_{1-\alpha} \bigg( 
\frac{\sum_{i=1}^{n+1} H(X_i,\tilde{X}_{n+1}) \cdot \delta_{S^y_i}}
{\sum_{i=1}^{n+1} H(X_i,\tilde{X}_{n+1})} \bigg) \bigg\}.  
\end{equation}
As $\tilde{X}_{n+1}$ is sampled to lie near $X_{n+1}$, we can view it as a proxy
for the test point, ensuring that RLCP will generally place higher weights on
data points near the test point.  

RLCP is known to have marginal coverage under exchangeability, and an
approximate form of conditional coverage.\footnote{\citet{hore2023conformal} 
  also derive a result on robustness to unknown covariate shift; this is
  possible to obtain via the unified framework as well but for a concise
  presentation we do not address this here.} 

\begin{theorem}[\citealt{hore2023conformal}]
\label{thm:rlcp}
If $Z_1,\dots,Z_{n+1}$ are exchangeable, then the RLCP set in \eqref{eqn:rlcp}
satisfies a marginal coverage guarantee $\PP{Y_{n+1}\in\Ccal(X_{n+1})}\geq
1-\alpha$. Moreover, for any $B\subseteq\Xcal$, RLCP satisfies an approximate
conditional coverage guarantee, 
\begin{align*}
\PPst{Y_{n+1}\in\Ccal(X_{n+1})}{X_{n+1}\in B} &\geq 1-\alpha \\ 
&\;\; - \frac{\inf_{\epsilon>0} \Big\{ \PPin{\|X_{n+1}-\tilde{X}_{n+1}\|>
  \epsilon} + \PPin{X_{n+1}\in \textnormal{bd}_{2\epsilon}(B)} \Big\}}
{\PP{X_{n+1}\in B}},
\end{align*}
where $\|\cdot\|$ is any norm on $\Xcal$, and where the $r$-boundary of a set
$B$ is defined as 
\[
\textnormal{bd}_r(B) = \bigg\{ x \in B: \inf_{x'\in B^c} \|x-x'\| \leq r
\bigg\}. 
\]
\end{theorem}

To interpret the second result, we see that as long as the kernel $H$ is
strongly localizing (i.e., $\|X_{n+1} - \tilde{X}_{n+1}\|$ is likely to be
small), an approximate coverage guarantee holds when we condition on $X_{n+1}\in
B$ for any set $B$ that satisfies some regularity conditions (so that its
boundary does not contain too much mass). In fact, the proof given later will
show that the unified framework allows us to establish a strictly stronger
result: 
\[
\PPst{Y_{n+1}\in\Ccal(X_{n+1})}{X_{n+1}\in B} \geq 1-\alpha -
\frac{\EE{\Var(\One{X_{n+1}\in B} \mid \tilde{X}_{n+1})}}{\PP{X_{n+1}\in B}}. 
\]

\subsubsection{View from the unified framework}

We describe how RLCP fits into the unified conformal framework.

\paragraph{Choices of $U$ and $Q_{Z\mid U}$.}

Define $U=(\Zbag,\tilde{X}_{n+1})$. Note that $U$ here contains additional
information beyond the bag of data. Overloading notation, we write the score
function for $u=(\zbag,\tilde{x})$ as
\[
\score(z,u) = \score(z_{n+1},\zbag).
\]
Next define, for $z=((x_1,y_1),\dots,(x_{n+1},y_{n+1}))$, 
\[
Q_{Z\mid U=(\zbag,\tilde{x})} = \frac{\sum_{\sigma\in\Scal_{n+1}}
  H(x_{\sigma(n+1)},\tilde{x}) \cdot \delta_{z_\sigma}}
{\sum_{\sigma\in\Scal_{n+1}} H(x_{\sigma(n+1)},\tilde{x})}. 
\]
This is the weighted empirical distribution which places weight proportional to
$H(x_{\sigma(n+1)},\tilde{x})$ on each $z_\sigma$ (i.e., higher weight if
$\tilde{X}_{n+1}=\tilde{x}$ has a higher likelihood given
$X_{n+1}=x_{\sigma(n+1)}$). 

\paragraph{P-value.}

Under these choices, the p-value in \eqref{eqn:unified_pvalue} is
\begin{align*}
p &= \frac{\sum_{\sigma\in\Scal_{n+1}} H(X_{\sigma(n+1)},\tilde{X}_{n+1}) 
  \cdot \One{\score(Z_{\sigma(n+1)},\Zbag) \geq
  \score(Z_{n+1},\Zbag)}}{\sum_{\sigma\in\Scal_{n+1}}
  H(X_{\sigma(n+1)},\tilde{X}_{n+1})} \\
&= \frac{\sum_{i=1}^{n+1} H(X_i,\tilde{X}_{n+1}) \cdot \One{\score(Z_i,\Zbag) 
  \geq \score(Z_{n+1},\Zbag)}}{\sum_{i=1}^{n+1} H(X_i,\tilde{X}_{n+1})}, 
\end{align*}
where the second line follows from a similar simplification as in the WCP
case. To see that this corresponds to RLCP, observe that by Lemma
\ref{lem:pvalue_vs_quantile},  
\[
p > \alpha \iff \score(Z_{n+1},\Zbag) \leq \Quantile_{1-\alpha}\bigg(
\frac{\sum_{i=1}^{n+1} H(X_i,\tilde{X}_{n+1}) \cdot
  \delta_{\score(Z_i,\Zbag)}}{\sum_{i=1}^{n+1} H(X_i,\tilde{X}_{n+1})} \bigg).  
\]
The right-hand side above can be directly seen to be equivalent to the event
$Y_{n+1} \in \Ccal(X_{n+1})$ for the RLCP set in \eqref{eqn:rlcp}. 

\paragraph{Validity.}

We prove Theorem \ref{thm:rlcp} using the unified result in Theorem
\ref{thm:unified}. 

\begin{proof}[Proof of Theorem \ref{thm:rlcp} via the unified framework]
Note that the marginal guarantee is implied by the approximate conditional
guarantee: we can simply take $B=\Xcal$ and let $\epsilon\to\infty$. Thus from
this point on, we focus on the conditional guarantee. 

First let $P^*_Z$ be the distribution of $Z=((X_1,Y_1),\dots,(X_{n+1},Y_{n+1}))$
(which we assume is exchangeable), and let $P_Z$ be the distribution of $Z$
conditional on the event $X_{n+1} \in B$. In other words, the conditional
coverage probability can be written as 
\[
\Ppst{P^*_Z}{Y_{n+1}\in\Ccal(X_{n+1})}{X_{n+1}\in B} = \Pp{P_Z}{Y_{n+1}\in
  \Ccal(X_{n+1})}. 
\]
Let $u=(\zbag,\tilde{x})$, and write $z=((x_1,y_1),\dots,(x_{n+1},y_{n+1}))$. 
Then 
\[
P_{Z\mid U=(\zbag,\tilde{x})} = \frac{\sum_{\sigma\in\Scal_{n+1}}
  H(x_{\sigma(n+1)},\tilde{x}) \cdot \One{x_{\sigma(n+1)}\in B} \cdot
  \delta_{z_\sigma}}{\sum_{\sigma\in\Scal_{n+1}} H(x_{\sigma(n+1)},\tilde{x})
  \cdot \One{x_{\sigma(n+1)} \in B}},
\]
by using the fact that $\d{P}_Z(z_\sigma) / \d{P}_Z(z) = \One{x_{\sigma(n+1)}\in 
  B}$ (for $x_{n+1}\in B$), and recalling that $\tilde{X}_{n+1}$ is drawn
from a distribution with density $H(X_{n+1},\cdot)$. Therefore,
\begin{align*}
\dtv(&P_{Z\mid U=(\zbag,\tilde{x})},Q_{Z\mid U=(\zbag,\tilde{x})}) \\ 
&\leq \sum_{\sigma\in\Scal_{n+1}} \Bigg( \frac{H(x_{\sigma(n+1)},\tilde{x}) \cdot 
  \One{x_{\sigma(n+1)}\in B}}{\sum_{\sigma'\in\Scal_{n+1}} H(x_{\sigma'(n+1)},
  \tilde{x}) \cdot \One{x_{\sigma'(n+1)}\in B}} - \frac{H(x_{\sigma(n+1)},
  \tilde{x})}{\sum_{\sigma'\in\Scal_{n+1}} H(x_{\sigma'(n+1)},\tilde{x})} 
  \Bigg)_+ \\
&= \sum_{i=1}^{n+1} \Bigg( \frac{H(x_i,\tilde{x}) \cdot \One{x_i\in B}}
  {\sum_{i'=1}^{n+1} H(x_{i'},\tilde{x}) \cdot \One{x_{i'}\in B}} - 
  \frac{H(x_i,\tilde{x})}{\sum_{i'=1}^{n+1 }H(x_{i'},\tilde{x})} \Bigg)_+ \\ 
&= \sum_{i=1}^{n+1} H(x_i,\tilde{x}) \cdot \One{x_i\in B} \cdot 
  \Bigg(\frac{1}{\sum_{i'=1}^{n+1} H(x_{i'},\tilde{x})\cdot \One{x_{i'}\in B}} 
  - \frac{1}{\sum_{i'=1}^{n+1} H(x_{i'},\tilde{x})} \Bigg) \\ 
&= 1 - \frac{\sum_{i=1}^{n+1} H(x_i,\tilde{x}) \cdot \One{x_i\in B}}
  {\sum_{i=1}^{n+1} H(x_i,\tilde{x})} = \frac{\sum_{i=1}^{n+1} H(x_i,\tilde{x})
  \cdot \One{x_i\not\in B}}{\sum_{i=1}^{n+1} H(x_i,\tilde{x})}. 
\end{align*}
Next let $P_U$ be the marginal distribution of $U=(\Zbag,\tilde{X}_{n+1})$ under
$P_{Z,U}$, and let $Q_U = P_U$. To bound $\dtv(P_{Z,U},Q_{Z,U})$, as $P_U=Q_U$,
we have $\dtv(P_{Z,U},Q_{Z,U}) = \Ep{P_U}{\dtv(P_{Z\mid U},Q_{Z\mid U})}$, and
so 
\begin{align*}
\dtv(P_{Z,U},Q_{Z,U}) 
&\leq \Ep{P_Z}{\frac{\sum_{i=1}^{n+1} H(X_i,\tilde{X}_{n+1}) \cdot 
  \One{X_i\not\in B}}{\sum_{i=1}^{n+1} H(X_i,\tilde{X}_{n+1})}} \\
&= \Epst{P^*_Z}{\frac{\sum_{i=1}^{n+1} H(X_i,\tilde{X}_{n+1}) \cdot 
  \One{X_i\not\in B}}{\sum_{i=1}^{n+1} H(X_i,\tilde{X}_{n+1})}}{X_{n+1}\in B}, 
\end{align*}
by definition of $P_Z$ as compared to $P^*_Z$. We can rewrite this as 
\begin{align*}
\dtv(P_{Z,U},Q_{Z,U})  
&= \frac{\Ep{P^*_Z}{\frac{\sum_{i=1}^{n+1} H(X_i,\tilde{X}_{n+1}) \cdot 
  \One{X_i\not\in B}}{\sum_{i=1}^{n+1} H(X_i,\tilde{X}_{n+1})} \cdot 
  \One{X_{n+1}\in B}}}{\Pp{P^*_Z}{X_{n+1}\in B}} \\
&= \frac{\Ep{P^*_Z}{\Ppst{P^*_Z}{X_{n+1}\not\in B}{\Zbag,\tilde{X}_{n+1}} \cdot 
  \Ppst{P^*_Z}{X_{n+1}\in B}{\Zbag,\tilde{X}_{n+1}}}}{\Pp{P^*_Z}{X_{n+1}\in B}}, 
\end{align*}
where the last step holds using similar calculations to the above to
characterize $Z \mid U$, this time applied to $X_{n+1} \mid U$. In particular,
using the variance of a Bernoulli random variable, we can see that 
\begin{align*}
\dtv(P_{Z,U},Q_{Z,U})  
&= \frac{\Ep{P^*_Z}{\Var_{P^*_Z}\left(\One{X_{n+1}\in B}\mid \Zbag,
  \tilde{X}_{n+1}\right)}}{\Pp{P^*_Z}{X_{n+1}\in B}} \\
&\leq \frac{\Ep{P^*_Z}{\Var_{P^*_Z}\left(\One{X_{n+1}\in B}\mid 
  \tilde{X}_{n+1}\right)}}{\Pp{P^*_Z}{X_{n+1}\in B}},
\end{align*}
where the inequality holds by the law of total variance.

To complete the proof, we need to show that this last expression results in the
bound claimed in the theorem. First, on the event $\tilde{X}_{n+1}\in B^c \cup
\textnormal{bd}_{\epsilon}(B)$, 
\[
\Ppst{P^*_Z}{X_{n+1}\in B}{\tilde{X}_{n+1}} \leq \Ppst{P^*_Z}{X_{n+1}\in
  \textnormal{bd}_{2\epsilon}(B) \;\, \textnormal{or} \;\,
  \|X_{n+1}-\tilde{X}_{n+1}\|>\epsilon}{\tilde{X}_{n+1}}. 
\]
On the other hand, on the event $\tilde{X}_{n+1}\in B\setminus
\textnormal{bd}_{\epsilon}(B)$, 
\[
\Ppst{P^*_Z}{X_{n+1}\not\in B}{\tilde{X}_{n+1}} \leq
  \Ppst{P^*_Z}{\|X_{n+1}-\tilde{X}_{n+1}\|>\epsilon}{\tilde{X}_{n+1}}.
\] 
Using $p(1-p) \leq \min\{p,1-p\}$, and combining the above bounds,  
\begin{align*}
\Var_{P^*_Z}&\left(\One{X_{n+1}\in B} \mid \tilde{X}_{n+1}\right) \\
&\qquad \leq \Ppst{P^*_Z}{X_{n+1} \in \textnormal{bd}_{2\epsilon}(B)\;\,
 \textnormal{or} \;\, \|X_{n+1}-\tilde{X}_{n+1}\|>\epsilon}{\tilde{X}_{n+1}} \\
&\qquad \leq \Ppst{P^*_Z}{X_{n+1}\in \textnormal{bd}_{2\epsilon}(B)}
  {\tilde{X}_{n+1}} + \Ppst{P^*_Z}{\|X_{n+1}-\tilde{X}_{n+1}\|>\epsilon}
  {\tilde{X}_{n+1}}.
\end{align*}
and marginalizing over $\tilde{X}_{n+1}$ completes the proof. 
\end{proof}

\subsection{Generalized weighted conformal prediction}
\label{sec:gwcp}

In two of the settings considered above---standard CP and WCP---the basic
premise is that, after conditioning on the unordered bag of data $\Zbag$, we
know the distribution of the data $Z$ itself. For standard CP, under
exchangeability, this distribution is uniform over the $(n+1)!$ possible
permutations of data vectors, while for WCP the distribution over permutations
is weighted due to distribution shift. A natural generalization is to consider a
setting where we allow the likelihood ratio between any two permutations of the
same dataset to be arbitrary, but still known; this leads to a generalization of
WCP, studied by \citet{prinster2024conformal}.

Specifically, suppose that $Z\in\Zcal^{n+1}$ has joint density
$f:\Zcal^{n+1}\to\R_+$, with respect to any exchangeable base measure. While
$f$ itself may be unknown, suppose that for any $z\in\Zcal^{n+1}$, we can
calculate the ratio  
\[
f(z_\sigma) / f(z).
\]
Returning to the earlier examples, under exchangeability this ratio is simply
equal to $1$, while in the WCP setting \eqref{eqn:wcp_wstar} this ratio is equal
to $w^*(z_{\sigma(n+1)}) / w^*(z_{n+1})$. 

As studied in \citet{nair2023randomization, prinster2024conformal} (building on
earlier work by \cite{fannjiang2022conformal}), an important application of this 
generalized view is the problem of feedback covariate shift (FCS), where data is
collected sequentially. At time $t$, the observed data from past times  
$1,\dots,t-1$ determines a distribution from which the next covariate $X_t$ is 
sampled---for example, this can arise if our aim is to identify regions of
feature space that are most likely to lead to, say, higher response
values. To make this concrete, suppose
\[
X_1 \sim g_{X,1}, \ Y_1 \mid X_1 \sim P_{Y\mid X=X_1},
\]
at time $t=1$, and then for each time $t \geq 2$,
\begin{align*}
X_t &\mid ((X_i,Y_i))_{i=1}^{t-1} \sim \alg\big( ((X_i,Y_i))_{i=1}^{t-1} \big), \\ 
Y_t &\mid \big(((X_i,Y_i))_{i=1}^{t-1}, X_t\big) \sim P_{Y\mid X=X_t}, 
\end{align*}
where $\alg$ is some algorithm mapping the observed data
$((X_i,Y_i))_{i=1}^{t-1}$ to a density for the next covariate $X_t$, relative to
some base measure on $\Xcal$. We can calculate the likelihood ratio by 
\begin{equation}
\label{eqn:fcs_lr}
\frac{f(z_\sigma)}{f(z)} = \frac{g_{X,1}(x_{\sigma(1)}) \cdot
  \prod_{t=2}^{n+1} \big[\alg(z_{\sigma(1)},\dots,z_{\sigma(t-1)})\big]
  (x_{\sigma(t)})}{g_{X,1}(x_1)\cdot \prod_{t=2}^{n+1}
  \big[\alg(z_1,\dots,z_{t-1})\big](x_t)},
\end{equation}
which does not depend on the unknown conditional distribution $P_{Y\mid
  X}$. In other words, here $f(z_\sigma)/f(z)$ is known, and depends only on
choices made by the user.   

\subsubsection{Method and theory}

For arbitrary $y \in \Ycal$, we define as usual $Z^y =
(Z_1,\dots,Z_n,(X_{n+1},y))$, and scores  
\[
S^y_i = \score(Z^y_i, \lbag Z^y\rbag), \quad i\in[n+1],
\]
for a  score function of the form $\score((x,y),\zbag)$. Next
define\footnote{For simplicity, we will assume that the density $f$ is always
  positive so that the ratio is well-defined, but the method can be modified in
  straightforward ways to avoid this condition.}  
\[
W^y_i = \sum_{\sigma : \sigma(n+1)=i} \frac{f(Z^y_\sigma)}{f(Z^y)}, 
\]
noting that by our assumptions above, each of these ratios is known even if the
joint density $f$ is unknown. The generalized WCP set at coverage level
$1-\alpha$ is then defined by
\begin{equation}
\label{eqn:gwcp}
\Ccal(X_{n+1}) = \bigg\{ y\in\Ycal : S^y_{n+1} \leq \Quantile_{1-\alpha}\bigg(
\frac{\sum_{i=1}^{n+1} W^y_i \cdot \delta_{S^y_i}}{\sum_{i=1}^{n+1} W^y_i}
\bigg)\bigg\}.  
\end{equation} 
(To reiterate, both standard CP and WCP can be seen as special cases of this
construction.) When $f(z_\sigma)/f(z)$ is known exactly, for any $z$ and
$\sigma$, generalized WCP has exact coverage. 

\begin{theorem}[\citealt{prinster2024conformal}]
\label{thm:gwcp}
If $Z\in\Zcal^{n+1}$ has density $f$ with respect to an exchangeable base
measure, then the generalized WCP set in \eqref{eqn:gwcp} satisfies $\PP{Y_{n+1}
  \in \Ccal(X_{n+1})} \geq 1-\alpha$. 
\end{theorem}

\subsubsection{View from the unified framework}

We describe how generalized WCP fits into the unified conformal framework.

\paragraph{Choices of $U$ and $Q_{Z\mid U}$.}

Define $U = \Zbag$, and overloading notation again, write the score function for
$u=\zbag$ as     
\[
\score(z,u) = \score(z_{n+1},\zbag).
\]
Next define 
\[
Q_{Z\mid U=\zbag} = \frac{\sum_{\sigma\in\Scal_{n+1}} f(z_\sigma) \cdot
    \delta_{z_\sigma}}{\sum_{\sigma\in\Scal_{n+1}} f(z_\sigma)}.
\]
Note that this can equivalently be written as
\[
Q_{Z\mid U=\zbag} = \frac{\sum_{\sigma\in\Scal_{n+1}} f(z_\sigma) / f(z) \cdot
    \delta_{z_\sigma}}{\sum_{\sigma\in\Scal_{n+1}} f(z_\sigma) / f(z)}.
\]
which can be computed by the analyst, since each ratio $f(z_\sigma)/f(z)$ is
assumed to be known.  

\paragraph{P-value.}

Under these choices, the p-value in \eqref{eqn:unified_pvalue} is 
\begin{align*}
p &= \frac{\sum_{\sigma\in\Scal_{n+1}} f(Z_\sigma) / f(Z) \cdot
  \One{\score(Z_{\sigma(n+1)},\Zbag) \geq \score(Z_{n+1},\Zbag)}}
  {\sum_{\sigma\in\Scal_{n+1}} f(Z_\sigma) / f(Z)} \\
&= \frac{\sum_{i=1}^{n+1} \sum_{\sigma:\sigma(n+1)=i} f(Z_\sigma) / f(Z) \cdot 
  \One{\score(Z_i,\Zbag) \geq \score(Z_{n+1},\Zbag)}}
  {\sum_{i=1}^{n+1} \sum_{\sigma:\sigma(n+1)=i} f(Z_\sigma) / f(Z)} \\ 
&= \frac{\sum_{i=1}^{n+1} W^{Y_{n+1}}_i \cdot
  \One{\score(Z_i,\Zbag) \geq \score(Z_{n+1},\Zbag)}}
  {\sum_{i=1}^{n+1} W^{Y_{n+1}}_i}.
\end{align*}
To see that this corresponds to generalized WCP, observe that by Lemma 
\ref{lem:pvalue_vs_quantile},  
\[
p>\alpha \iff \score(Z_{n+1},\Zbag) \leq \Quantile_{1-\alpha}\bigg(
\frac{\sum_{i=1}^{n+1} W^{Y_{n+1}}_i \cdot \delta_{\score(Z_i,\Zbag)}} 
  {\sum_{i=1}^{n+1} W^{Y_{n+1}}_i} \bigg). 
\]  
The right-hand side above can be directly seen to be equivalent to the event
$Y_{n+1} \in \Ccal(X_{n+1})$, for the generalized WCP set in \eqref{eqn:gwcp}. 

\paragraph{Validity.}

We prove Theorem \ref{thm:gwcp} using the unified result in Theorem
\ref{thm:unified}. 

\begin{proof}[Proof of Theorem \ref{thm:gwcp} via the unified framework]
Recall that $U=\Zbag$. Since we have assumed that $Z$ has joint density $f$ with
respect to some exchangeable base measure, this means that $P_{Z\mid U} =
Q_{Z\mid U}$, by construction. Therefore (recalling Remark
\ref{rmk:exact_validity}), we have $\PP{p\leq \alpha}\leq \alpha$.  
\end{proof}

\section{Extensions of the unified framework}
\label{sec:unified_ext}

In this section, we develop several extensions of the unified framework. These
extensions will be used to derive some of the new results on conformal given in
the next section. All proofs for this section are deferred until the appendix. 

\subsection{Monte Carlo p-values}
\label{sec:unified_mc}

For our first extension, we develop Monte Carlo versions of Theorem
\ref{thm:unified}, which can be useful in settings where the computational cost
of conformal procedures is prohibitive. Specifically, in using Monte Carlo,
when computing the p-value $p$ we can avoid integration with respect to $Q_{Z
  \mid U}$ (which may be computationally hard) by instead sampling with
respect to $Q_{Z \mid U}$ (which may be easier). 

\subsubsection{Simple Monte Carlo}

Given the conditional distribution $Q_{Z\mid U}$, the p-value $p$ defined
in \eqref{eqn:unified_pvalue} for the unified method requires integration with
respect to $Q_{Z\mid U}$: we can rewrite this p-value as 
\[
p = \int_\Zcal \One{\score(z,U) \geq \score(Z,U)} \, \d{Q}_{Z\mid U}(z). 
\]
To avoid calculating such an integral (i.e., to avoid calculating a probability
with respect to $Q_{Z\mid U}$), we can instead take a simple Monte Carlo
approximation,     
\begin{equation}
\label{eqn:unified_pvalue_mc}
\tilde{p} = \frac{1}{M+1} \sum_{m=0}^M \One{\score(Z^{(m)},U) \geq \score(Z,U)},   
\end{equation}
where $Z^{(1)},\dots,Z^{(M)}$ are i.i.d.\ samples from $Q_{Z\mid U}$, with
$(Z^{(1)},\dots,Z^{(M)}) \independent Z \mid U$, and where we let $Z^{(0)} = Z$
for notational simplicity. (Observe that, since $Q_{Z\mid U}$ is supported on
the set $\{z \in \Zcal:\zbag = h(U) = \Zbag\}$, each $Z^{(m)}$ must be equal to
some permutation of $Z$.)

Next we give our guarantee for this Monte Carlo p-value $\tilde{p}$, similar to
that in Theorem \ref{thm:unified} for the original p-value $p$ in
\eqref{eqn:unified_pvalue}. This result is proved in Appendix
\ref{app:proof_thm:unified_mc}.

\begin{theorem}
\label{thm:unified_mc}
Suppose $(Z,U) \sim P_{Z,U}$. Then the Monte Carlo p-value defined in
\eqref{eqn:unified_pvalue_mc} satisfies   
\[
\PP{\tilde{p} \leq \alpha} \leq \alpha + \inf_{Q_U} \, \dtv(P_{Z,U},Q_{Z,U}),
\]
where the infimum is taken over all distributions $Q_U$ on $U$, and where
$Q_{Z,U} = Q_{Z\mid U}\times Q_U$.
\end{theorem}

\begin{remark}
\label{rmk:information_monotonicity_mc}
The guarantee above, with Type I error inflation bounded by
$\dtv(P_{Z,U},Q_{Z,U})$, is exactly the same as that derived in Remark
\ref{rmk:information_monotonicity}, which is a corollary to Theorem
\ref{thm:unified}. Indeed the original result of Theorem \ref{thm:unified}, with
Type I error inflation bounded by $\dtv(P_{S,T},Q_{S,T})$, is strictly stronger:
as $(S,T)$ may potentially contain much less information than $(Z,U)$, this TV
distance can be substantially smaller. An analogous refinement can also be
derived for the Monte Carlo case, but we defer this to the appendix; see
Appendix \ref{app:unified_mc_refine}.
\end{remark}

\begin{remark}
\label{rmk:exact_validity_mc}
Just as in Remark \ref{rmk:exact_validity}, if $Q_{Z \mid U} = P_{Z \mid U}$
holds almost surely (i.e., in implementing the Monte Carlo method we correctly
specify the conditional distribution of $Z \mid U$), then by taking $Q_U=P_U$ we
obtain exact Type I error control,       
\[
\PP{\tilde{p} \leq \alpha} \leq \alpha.
\]
\end{remark}

\subsubsection{Importance sampling}

When sampling from $Q_{Z \mid U}$ is itself a difficult task, we can also employ
alternate Monte Carlo schemes based on importance sampling. Let $R_{Z \mid U}$
be any (conditional) proposal distribution. We will assume that $Q_{Z \mid U},
R_{Z \mid U}$ are absolutely continuous with respect to each other, almost 
surely. Given a number of samples $M \geq 1$, and letting $Z^{(0)} = Z$ as
before, we now study an importance sampling scheme to define the p-value:  
\begin{equation}
\label{eqn:unified_pvalue_is}
\tilde{p} = \frac{\sum_{m=0}^M \frac{\d{Q}_{Z \mid U}(Z^{(m)})}{\d{R}_{Z \mid 
  U}(Z^{(m)})} \cdot \One{\score(Z^{(m)},U) \geq \score(Z,U)}}{\sum_{m=0}^M 
\frac{\d{Q}_{Z \mid U}(Z^{(m)})}{\d{R}_{Z \mid U}(Z^{(m)})}},
\end{equation}
where now $Z^{(1)},\dots,Z^{(M)}$ are i.i.d.\ samples from $R_{Z\mid U}$, and
$(Z^{(1)},\dots,Z^{(M)}) \independent Z \mid U$. In the Monte Carlo literature,
this type of construction is often called \emph{self-normalized} importance
sampling. Notice that this can be viewed as a generalization of the Monte Carlo 
procedure above, as the p-value $\tilde{p}$ in \eqref{eqn:unified_pvalue_mc} can
be obtained by simply letting $R_{Z\mid U} = Q_{Z\mid U}$.

We now present our theoretical guarantee for the importance sampling p-value
$\tilde{p}$. This result is proved in Appendix \ref{app:proof_thm:unified_is}. 

\begin{theorem}
\label{thm:unified_is}
Suppose $(Z,U) \sim P_{Z,U}$. Then the importance sampling p-value defined in
\eqref{eqn:unified_pvalue_is} satisfies    
\[
\PP{\tilde{p} \leq \alpha}\leq \alpha + \inf_{Q_U} \, \dtv(P_{Z,U},Q_{Z,U}),
\]
where the infimum is taken over all distributions $Q_U$ on $U$, and where
$Q_{Z,U} = Q_{Z\mid U}\times Q_U$. 
\end{theorem}

\begin{remark}
\label{rmk:information_monotonicity_is}
Analogous to Remark \ref{rmk:information_monotonicity_mc} for the simple Monte
Carlo setting, refinements of this result are again possible, with an inflation
in Type I error bounded by a TV distance between the (joint) distributions of
scores and threshold, but we omit the details.
\end{remark}

\begin{remark}
\label{rmk:exact_validity_is}
As in Remark \ref{rmk:exact_validity_mc} for simple Monte Carlo, if $Q_{Z\mid U}
= P_{Z\mid U}$, almost surely, then we obtain exact Type I error control,
\[
\PP{\tilde{p} \leq \alpha} \leq \alpha.
\] 
\end{remark}

\subsection{Unnormalized $Q_{Z \mid U}$}
\label{sec:unified_un}

For our second extension, we will develop an unnormalized version of Theorem
\ref{thm:unified}. Returning to the construction of the p-value, we now allow
$Q_{Z \mid U = u}$ to be measure on $\Zcal$, for each $u\in\Ucal$; the
difference here is that $Q_{Z\mid U=u}$ is no longer required to be a
distribution---that is, we may have $Q_{Z\mid U=u}(\Zcal)\neq 1$. However, as
before, we assume that $Q_{Z\mid U=u}$ is supported on the finite set
$\{z\in\Zcal^{n+1} : \zbag = h(u)\}$, for each $u$.

We then define   
\begin{equation}
\label{eqn:unified_pvalue_un}
p = p(Z,U) \quad \text{where} \quad p(z,u) = \int_\Zcal \One{\score(z',u) \geq 
  \score(z,u)} \, \d{Q}_{Z \mid U=u}(z'). 
\end{equation}
If $Q_{Z \mid U=u}$ is a distribution for each $u \in \Ucal$, then note this is
the same as the original p-value constructed in \eqref{eqn:unified_pvalue}. 
In the general case, we still refer to \eqref{eqn:unified_pvalue_un} as a 
p-value. We will also define a generalized threshold 
\[
\thresh_\alpha(u) = \inf \bigg\{ s \in \R : \int_\Zcal \One{\score(z,u) > s}
\, \d{Q}_{Z \mid U=u}(z) \leq \alpha \bigg\},
\]
which reduces to $\thresh_\alpha(u) = \Quantile_{1-\alpha}(Q_{S\mid U=u})$ when
$Q_{Z \mid U=u}$ is a distribution. Lastly, given a distribution $Q_U$ on
$\Ucal$, we write $Q_{Z,U} = Q_{Z\mid U}\times Q_U$ for the measure on
$(Z,U)\in\Zcal\times\Ucal$ that is defined as\footnote{Formally, in order for
  this to be well-defined, we require that the function $u\mapsto Q_{Z\mid
    U=u}(\{z :(z,u)\in A\})$ is measurable, for each measurable 
  $A\subseteq\Zcal\times\Ucal$; this condition is naturally satisfied in the
  normalized case, where $Q_{Z\mid U}$ is a conditional distribution.}  
\[
Q_{Z,U}(A) = \Ep{Q_U}{\int_\Zcal \One{(z,U)\in A}\;\d{Q_{Z\mid U}}(z)},
\]
for measurable $A \subseteq \Zcal\times\Ucal$. In the normalized case, where
$Q_{Z\mid U}$ is chosen to be a conditional distribution, note that this is
simply the joint distribution with marginal $Q_U$ and conditional $Q_{Z\mid U}$,
as before.

We are now ready to present the unnormalized generalization of Theorem
\ref{thm:unified}. Its proof is similar to that of Theorem \ref{thm:unified},
and is given in Appendix \ref{app:unified_un}.  

\begin{theorem}
\label{thm:unified_un}
Suppose $(Z,U) \sim P_{Z,U}$. Let $P_{S,T}$ be the induced joint distribution on
$(S,T)=(\score(Z,U),\thresh_\alpha(U))$. Then the p-value defined in
\eqref{eqn:unified_pvalue_un} satisfies   
\[
\PP{p \leq \alpha} \leq \alpha + \inf_{Q_U} \, \sup_{A\subseteq\R\times\R} \,
\Big\{ P_{S,T}(A) - Q_{S,T}(A) \Big\},   
\]
where the infimum is taken over all distributions $Q_U$ on $U$, and where
$Q_{S,T}$ denotes the joint measure on
$(S,T)=(\score(Z,U),\thresh_\alpha(U))\in\R\times\R$ induced by the joint
measure $Q_{Z,U} = Q_{Z\mid U}\times Q_U$ on $(Z,U)\in\Zcal\times\Ucal$. 
\end{theorem}

When $Q_{Z \mid U}$ is a conditional distribution, this reduces to the bound in
Theorem \ref{thm:unified}, since then $Q_{S,T}$ is a distribution and so we have
$\dtv(P_{S,T},Q_{S,T}) = \sup_{A\subseteq\R\times\R} \{P_{S,T}(A) -
Q_{S,T}(A)\}$. Before moving on to discuss new results in the next section, we 
make one further remark about information monotonicity. 

\begin{remark}
\label{rmk:information_monotonicity_un}
Just as before in Remark \ref{rmk:information_monotonicity}, information
monotonicity implies simpler but coarser versions of the bound in Theorem
\ref{thm:unified_un}. For example, because $(S,T)$ is a function of $(Z,U)$, we  
have the weaker bound
\[
\PP{p \leq \alpha} \leq \alpha + \inf_{Q_U} \, \sup_{A\subseteq\Zcal\times\Ucal}
\, \Big\{P_{Z,U}(A) - Q_{Z,U}(A)\Big\}. 
\]
\end{remark}

\section{Special cases: new results}
\label{sec:cases_new}

We now present new results which can be derived as special cases of the unified
framework for conformal prediction. We will focus on developing extensions of the 
conformal approaches and theory in Section \ref{sec:cases_known}; our intention
to provide a flavor of the types of extensions possible and not to derive an 
exhaustive set of new results. Some of these examples will rely on the
extensions of the unified theory from Section \ref{sec:unified_ext}. 

\subsection{WCP with unnormalized weights}
\label{sec:wcp_un}

We return to weighted conformal prediction (WCP) as considered in Section
\ref{sec:wcp}, for handling distribution shift. As before, we will assume that
the training data $Z_1,\dots,Z_n$ are  i.i.d.\ from $\Ptr$, the test point is
drawn independently from $\Pte$, and we have access to a weight function $w$
that approximates the likelihood ratio in \eqref{eqn:wcp_wstar}. Recall, the WCP
set \eqref{eqn:wcp} is defined by placing a weight on point $i$ which is
proportional to $w(Z_i)$, where the weights are normalized to sum to 1. Here we
study an alternative construction, where the weights are applied without such a  
normalization step, and we will derive a validity guarantee using the
unnormalized theory from Section \ref{sec:unified_un}.  

\subsubsection{Method and theory}

Fix any score function $\score((x,y),\zbag)$. Defining the data points $Z^y_i$
and scores $S^y_i$ exactly as in Section \ref{sec:wcp}, the unnormalized WCP set
at level $1-\alpha$ is defined by 
\begin{equation}
\label{eqn:wcp_un}
\Ccal(X_{n+1}) = \Bigg\{ y \in \Ycal: S^y_{n+1} \leq \inf\bigg\{s\in\R :
\frac{\sum_{i=1}^{n+1} w(Z^y_i)\cdot\One{S^y_i>s}  }{n+1}\leq \alpha\bigg\}
\Bigg\}.   
\end{equation}
To compare to the usual (normalized) WCP method, we can observe that the
original WCP set \eqref{eqn:wcp} can equivalently be defined as 
\begin{equation}
\label{eqn:wcp_alt}
\Ccal(X_{n+1}) = \Bigg\{ y \in \Ycal: S^y_{n+1} \leq \inf\bigg\{s\in\R :
\frac{\sum_{i=1}^{n+1} w(Z^y_i)\cdot\One{S^y_i>s}}{\sum_{i=1}^{n+1}w(Z^y_i)}
\leq \alpha\bigg\} \Bigg\},  
\end{equation}
since this infimum is simply equal to the quantile,
\begin{align*}
&\inf\bigg\{s\in\R : \frac{\sum_{i=1}^{n+1} 
  w(Z^y_i)\cdot\One{S^y_i>s}}{\sum_{i=1}^{n+1}w(Z^y_i)}  \leq \alpha\bigg\} \\
&= \inf\bigg\{s\in\R : \frac{\sum_{i=1}^{n+1} w(Z^y_i)\cdot\One{S^y_i\leq s}} 
  {\sum_{i=1}^{n+1}w(Z^y_i)}  \geq 1-\alpha\bigg\} 
= \Quantile_{1-\alpha}\bigg(\frac{\sum_{i=1}^{n+1} w(Z^y_i)\cdot 
  \delta_{S^y_i}}{\sum_{i=1}^{n+1}w(Z^y_i)}\bigg).
\end{align*}
The only difference between the unnormalized WCP method in \eqref{eqn:wcp_un},
and the original WCP method as expressed in \eqref{eqn:wcp_alt}, is that the
former divides each weight $w(Z^y_i)$ by $n+1$ rather than
\smash{$\sum_{j=1}^{n+1} w(Z^y_j)$}. To motivate this, notice that for the true
likelihood ratio $w^* = \d\Pte/\d\Ptr$, we have $\EE{w^*(Z_i)}=1$ for each  
$i\in[n]$, and so the original denominator \smash{$\sum_{j=1}^{n+1}w(Z^y_j)$}
has expected value approximately $n+1$. 

The unnormalized WCP method has guarantees analogous to that for WCP.

\begin{theorem}
\label{thm:wcp_un}
For independent $Z_1,\dots,Z_n \sim \Ptr$ and $Z_{n+1} \sim \Pte$, where as
before $w^* = \d\Pte/\d\Ptr$, the unnormalized WCP set in \eqref{eqn:wcp_un}
satisfies the following:  

\begin{itemize}[align=left,labelwidth={1.25em}]
\item[(a)] If $w = w^*$, then $\PP{Y_{n+1} \in \Ccal(X_{n+1})} \geq 1-\alpha$. 

\item[(b)] For any $w$, we have $\PP{Y_{n+1} \in \Ccal(X_{n+1})} \geq 1-\alpha -  
\Ep{\Ptr}{(w^*(X,Y) - w(X,Y))_+}$.
\end{itemize}
\end{theorem}

To take a closer look at part (b) of this result, note also that if
$\Ep{\Ptr}{w(X,Y)} = 1$ (the user-specified likelihood ratio estimate $w$ is
normalized at the population level), then
\[
\PP{Y_{n+1} \in \Ccal(X_{n+1})} \geq 1-\alpha - \frac{1}{2}\Ep{\Ptr}{|w^*(X,Y) -
  w(X,Y)|},
\] 
since, for any random variable $V\in\R$ with $\EE{V}=0$, we have $\EE{(V)_+} =
\EE{(V)_-} = \frac{1}{2}\EE{|V|}$. This result now appears identical to the
guarantee of Theorem \ref{thm:wcp} part (b), for the normalized case. But there 
is a subtle difference: here $w$ must already be normalized (at the
population), whereas for normalized WCP, the result always holds (since it
applies to $\bar{w}$ in place of $w$).    

\subsubsection{View from the unified framework}

We describe how unnormalized WCP fits into the unified conformal framework.

\paragraph{Choices of $U$ and $Q_{Z\mid U}$.}

Define $U = \Zbag$, and write the score function for $u=\zbag$ as    
\[
\score(z,u) = \score(z_{n+1},\zbag).
\]
Next define 
\[
Q_{Z \mid U=\zbag} = \frac{\sum_{\sigma \in \Scal_{n+1}} w(z_{\sigma(n+1)})
  \cdot \delta_{z_\sigma}}{(n+1)!},
\]
which is a measure (not necessarily a distribution), supported on permutations
of $z$.

\paragraph{P-value.}

Under these choices, the p-value in \eqref{eqn:unified_pvalue_un} is
\begin{align*}
p &= \frac{\sum_{\sigma \in \Scal_{n+1}} w(Z_{\sigma(n+1)}) \cdot
  \One{\score(Z_{\sigma(n+1)}, \Zbag) \geq \score(Z_{n+1},
  \Zbag)}}{(n+1)!} \\ 
&= \frac{\sum_{i=1}^{n+1} w(Z_i) \cdot \One{\score(Z_i, \Zbag) \geq 
  \score(Z_{n+1}, \Zbag)}}{n+1},
\end{align*}
where the second line follows from the fact that, for each $i$, there are $n!$
many permutations $\sigma\in\Scal_{n+1}$ with $\sigma(n+1)=i$. To see that this
corresponds to the unnormalized WCP set, we have 
\[
p > \alpha \iff \score(Z_{n+1}, \Zbag) \leq \inf\bigg\{s\in\R :
  \frac{\sum_{i=1}^{n+1} w(Z_i) \cdot \One{\score(Z_i, \Zbag)>s}}{n+1} \leq 
  \alpha\bigg\}. 
\]
by Lemma \ref{lem:p_vs_q_measure}, which is an unnormalized analogue of Lemma 
\ref{lem:pvalue_vs_quantile}, and is given in Appendix \ref{app:unified_un}. The 
right-hand side above can be directly seen to be equivalent to the event
$Y_{n+1} \in \Ccal(X_{n+1})$, for the unnormalized WCP set in
\eqref{eqn:wcp_un}.  

\paragraph{Validity.} 

We prove Theorem \ref{thm:wcp_un} using the unified result in Theorem
\ref{thm:unified_un}. 

\begin{proof}[Proof of Theorem \ref{thm:wcp_un} via the unified framework] 
As in the proof of Theorem \ref{thm:wcp}, part (a) is a special case of part
(b), so we only prove part (b).  First we define $Q_U$ as the distribution of
$U=\Zbag$ when $Z\sim \Ptr^{n+1}$. We can then calculate the measure $Q_{Z,U}$
as 
\begin{align*}
Q_{Z,U}(A) &=\Ep{Q_U}{\int_\Zcal \One{(z,U)\in A}\d{Q_{Z\mid U}}(z)} \\
&=\Ep{\Ptr^{n+1}}{\frac{\sum_{\sigma\in\Scal_{n+1}}w(Z_{\sigma(n+1)})\cdot
  \One{(Z_\sigma,\Zbag)\in A}}{(n+1)!}}\\ 
&=\frac{\sum_{\sigma\in\Scal_{n+1}}\Ep{\Ptr^{n+1}}{w(Z_{\sigma(n+1)})\cdot
  \One{(Z_\sigma,\Zbag)\in A}}}{(n+1)!}\\
&=\Ep{\Ptr^{n+1}}{w(Z_{n+1})\cdot\One{(Z,\Zbag)\in A}},
\end{align*}
where the second step holds by definition of $Q_U$ and $Q_{Z\mid U}$, and the
last step holds since $\Ptr^{n+1}$ is an exchangeable distribution. On the other
hand, since $P_Z = \Ptr^n\times (\Ptr\circ w^*)$, we have 
\[
P_{Z,U}(A) = \Pp{\Ptr^n\times (\Ptr\circ w^*)}{(Z,\Zbag)\in A}  =
\Ep{\Ptr^{n+1}}{w^*(Z_{n+1})\cdot\One{(Z,\Zbag)\in A}}.
\] 
Then
\begin{align*}
\sup_{A\subseteq\Zcal\times\Ucal}\, \{ P_{Z,U}(A) - Q_{Z,U}(A) \} 
&= \sup_A \, \Ep{\Ptr^{n+1}}{(w^*(Z_{n+1}) - w(Z_{n+1}))\cdot
  \One{(Z,\Zbag)\in A}} \\
&\leq \sup_A \, \Ep{\Ptr^{n+1}}{(w^*(Z_{n+1}) - w(Z_{n+1}))_+}.
\end{align*}
Applying Theorem \ref{thm:unified_un} (together with Remark
\ref{rmk:information_monotonicity_un}) completes the proof. 
\end{proof}

\subsection{WCP under distribution drift}

When we studied WCP previously in Section \ref{sec:wcp}, and then again in 
Section \ref{sec:wcp_un}, our model for the data was that the training points
$Z_1,\dots,Z_n$ are drawn i.i.d.\ from some distribution $\Ptr$, while the
test point $Z_{n+1}$ is instead drawn from $\Pte$. In this subsection, we will
extend to a more challenging setting: we will allow for the training data to
exhibit distribution drift, as well. 
To formalize this, suppose $Z_i \sim \Ptr_i$, $i \in [n]$, and $Z_{n+1} \sim
\Pte$, all independently. In other words, there is distribution shift between
training and test sets, and in addition, the training set consists of
independent but not necessarily identically distributed samples.

\subsubsection{Method and theory}

The method is exactly the same as before, in Section \ref{sec:wcp}---fixing any 
score function of the form $\score((x,y),\zbag)$, we will consider the WCP set
in \eqref{eqn:wcp}, with a prespecified weight function $w(x,y)$. At a high
level, the idea is that WCP should achieve approximate coverage as long as the
training data are approximately i.i.d.\ (i.e., $F_1,\dots,F_n$ are similar each
other), and $w$ approximates the distribution shift (i.e., for each $i$, we have
$w\approx \d\Pte/\d\Ptr_i$). 

As our next result shows, WCP has approximate coverage in this setting with 
training distribution drift.    

\begin{theorem}
\label{thm:wcp_drift}
For independent $Z_1\sim\Ptr_1,\dots,Z_n\sim\Ptr_n,Z_{n+1}\sim\Pte$,
assuming $\Ep{\bar{\Ptr}}{w(X,Y)} < \infty$ under the mixture $\bar{\Ptr} =  
\frac{1}{n}\sum_{i=1}^n \Ptr_i$, the WCP set in \eqref{eqn:wcp} satisfies  
\[ 
\PP{Y_{n+1}\in\Ccal(X_{n+1})} \geq 1-\alpha -  \frac{1}{n}\sum_{i=1}^n
\Big\{\Ep{\Ptr_i}{\big(w^*_i(X,Y) - \bar{w}(X,Y)\big)_+} +
\dtv(\Ptr_i,\bar{\Ptr})\Big\},\]  
where we define 
\[
w^*_i(x,y) = \frac{\d{\Pte}}{\d{\Ptr_i}}(x,y), \quad i \in [n],
\]
and where $\bar{w}(x,y) = w(x,y)/\Ep{\bar{\Ptr}}{w(x,y)}$.
\end{theorem}

Comparing this to the original result for WCP, we can see that the bound in 
Theorem \ref{thm:wcp} can be derived as a special case with $\Ptr_i=\Ptr$ and
consequently $w^*_i = w$, for each $i$.

\subsubsection{View from the unified framework}

We describe how WCP under training drift fits into the unified conformal
framework. 

\paragraph{Choices of $U$ and $Q_{Z\mid U}$.}

The unified view here is the same in Section \ref{sec:wcp}, for WCP. For
concreteness, we set $U=\Zbag$, and define  
\[
Q_{Z \mid U=\zbag} = \frac{\sum_{\sigma \in \Scal_{n+1}} w(z_{\sigma(n+1)})
  \cdot \delta_{z_\sigma}}{\sum_{\sigma \in \Scal_{n+1}} w(z_{\sigma(n+1)})}.
\]

\paragraph{P-value.}

Since the choice of $U$ and $Q_{Z\mid U}$ is exactly the same as in
Section \ref{sec:wcp}, we again have that $p>\alpha$ if and only if 
$Y_{n+1}\in\Ccal(X_{n+1})$, for the WCP set defined in \eqref{eqn:wcp}. 

\paragraph{Validity.}

While the construction of the method is identical to that in Section
\ref{sec:wcp} (i.e., we are simply using the same WCP method), the proof is
now different because our assumption on the data is more general, allowing for
drift.  

\begin{proof}[Proof of Theorem \ref{thm:wcp_drift} via the unified framework]
We choose $Q_U$ to be the distribution of $U=\Zbag$ when $Z \sim H$, where   
\[
\d{H}(z) =  \frac{1}{n+1} \sum_{i=1}^{n+1}\bar{w}(z_i) \cdot
\d(F_1\times\dots\times F_n\times \bar{F})(z). 
\]
(Note that $H$ is a distribution---i.e., the above definition integrates to
$1$---by definition of $\bar{F}$ and $\bar{w}$.) The result of
Theorem \ref{thm:unified} (together with Remark
\ref{rmk:information_monotonicity}) tells us that  
\[
\PP{p\leq\alpha}\leq \alpha + \dtv(P_{Z_{n+1},U},Q_{Z_{n+1},U}),
\]
because $(S,T)$ can be expressed as a function of $(S,U)$, which in
turn can be expressed as a function of $(Z_{n+1},U)$. To be clear, here
$Q_{Z_{n+1}, U}$ denotes the distribution of $(Z_{n+1}, U)$ induced by the joint
distribution $Q_{Z,U}=Q_{Z\mid U}\times Q_U$. 

We now need to bound the total variation distance, which we can express as
\[
\dtv(P_{Z_{n+1},U},Q_{Z_{n+1},U}) = \sup_{A\in\Zcal\times\Ucal} \,
\Big\{P_{Z_{n+1},U}(A) - Q_{Z_{n+1},U}(A)\Big\}.
\] 
First, observe that by definition of $Q_U$ and $Q_{Z\mid U}$, 
\begin{align*}
&Q_{Z_{n+1},U}(A)=\Pp{Q_{Z,U}}{(Z_{n+1},U)\in A}\\
&=\Ep{H}{\sum_{\sigma\in\Scal_{n+1}}\frac{w(Z_{\sigma(n+1)})}
{\sum_{\sigma'\in\Scal_{n+1}}w(Z_{\sigma'(n+1)})} 
\One{(Z_{\sigma(n+1)},\Zbag)\in A}} \\
&=\Ep{H}{\sum_{i=1}^{n+1}\frac{w(Z_i)}
{\sum_{j=1}^{n+1}w(Z_j)}\One{(Z_i,\Zbag)\in A}} \\
&=\Ep{\Ptr_1\times\dots\times \Ptr_n\times \bar{\Ptr}}
{\frac{\sum_{i=1}^{n+1}\bar{w}(Z_i)}{n+1}\cdot 
\sum_{i=1}^{n+1}\frac{w(Z_i)}{\sum_{j=1}^{n+1}w(Z_j)}
\One{(Z_i,\Zbag)\in A}} \\
&=\Ep{\Ptr_1\times\dots\times \Ptr_n\times \bar{\Ptr}}
{\frac{1}{n+1} \sum_{i=1}^{n+1}\bar{w}(Z_i)\cdot \One{(Z_i,\Zbag)\in A}},
\end{align*}
where the next-to-last step uses the definition of $H$. Therefore, we can split
the TV bound into two types of terms:  
\begin{multline*}
\dtv(P_{Z_{n+1},U},Q_{Z_{n+1},U}) \leq {}\\
\frac{1}{n+1}\sum_{i=1}^n \underbrace{\sup_A\,\Big\{P_{Z_{n+1},U}(A)
  -\Ep{\Ptr_1\times\dots\times \Ptr_n\times \bar{\Ptr}}{\bar{w}(Z_i)\cdot
  \One{(Z_i,\Zbag)\in A}}\Big\}}_{\textnormal{Term $i$}} \\
{}+ \frac{1}{n+1} \underbrace{\sup_A\,\Big\{P_{Z_{n+1},U}(A) -
  \Ep{\Ptr_1\times\dots\times \Ptr_n\times \bar{\Ptr}}{\bar{w}(Z_{n+1})\cdot
  \One{(Z_{n+1},\Zbag)\in A}}\Big\}}_{\textnormal{Term $n+1$}}.
\end{multline*} 
In Appendix \ref{app:wcp_drift}, we will verify that
\begin{equation}
\label{eqn:wcp_drift_term_i}
\textnormal{Term $i$} \leq 
\Ep{\Ptr_i}{\big(w^*_i(X,Y)-\bar{w}(X,Y)\big)_+} +\dtv(\Ptr_i,\bar{\Ptr}),
\end{equation}
for each $i\in[n]$, and
\begin{equation}
\label{eqn:wcp_drift_term_n+1}
\textnormal{Term $n+1$} \leq
\frac{1}{n}\sum_{i=1}^n \Ep{\Ptr_i}{\big(w^*_i(X,Y) - \bar{w}(X,Y)\big)_+}.
\end{equation}
Combining these calculations, we then have
\[
\dtv(P_{Z_{n+1},U},Q_{Z_{n+1}\mid U}\times Q_U) \leq \frac{1}{n}\sum_{i=1}^n
\Big\{\Ep{\Ptr_i}{\big(w^*_i(X,Y) - \bar{w}(X,Y)\big)_+} +
\dtv(\Ptr_i,\bar{\Ptr})\Big\},
\]
which completes the proof.
\end{proof}

\subsection{RLCP with feature resampling}

We now return to randomly-localized conformal prediction (RLCP), introduced in
Section \ref{sec:rlcp}. Recall, in that method, after sampling a noisy version
of the test feature $\tilde{X}_{n+1} \sim H(X_{n+1},\cdot)$, the conformal set
is defined by placing weight $\propto H(X_i,\tilde{X}_{n+1})$ on each data   
point $i\in[n+1]$, placing higher weight on data lying near the test point. Here
we introduce a different version of the method: the weights will be centered at
$X_K$ for some random choice of $K\in[n+1]$, rather than at a synthetic value 
$\tilde{X}_{n+1}$. This version of RLCP may be computationally easier if
sampling from the density $H(X_{n+1},\cdot)$ is difficult. 

Furthermore, we may prefer a version of RLCP which samples among existing
observed covariate values since this may be more interpretable in certain
settings. For example, if our feature points lie on a lattice, then directly
sampling $\tilde{X}_{n+1}$ may not be a feasible value for the test
feature---RLCP with a Gaussian kernel $H$ leads to a value $\tilde{X}_{n+1}$
that lies off the lattice. In contrast, the modified form of RLCP presented here
will always reuse an existing data point instead of sampling a new value, and
therefore avoids this issue.

\subsubsection{Method and theory}

Fix any score function of the form $\score((x,y),\zbag)$, and a localizing
kernel $H:\Xcal\times\Xcal\to\R_+$, as in RLCP in Section \ref{sec:rlcp}.
Now, given the training features $X_1,\dots,X_n$ and test feature $X_{n+1}$, we 
sample a random index $K\in[n+1]$ by 
\[
K \mid X_1,\dots,X_{n+1} \sim \sum_{k=1}^{n+1} \bar{H}_{n+1,k}\cdot \delta_k, 
\] 
where for $i,k \in [n+1]$, we introduce the weight
\[
\bar{H}_{ik} = \frac{H(X_i,X_k)}{\sum_{j=1}^{n+1} H(X_i,X_j)}.
\]
Defining $Z^y$ and $S^y_i$, $i\in[n+1]$ as before, we then define the modified
RLCP set as 
\begin{equation}
\label{eqn:rlcp_mod}
\Ccal(X_{n+1}) = \bigg\{ y \in \Ycal: S^y_{n+1} \leq \Quantile_{1-\alpha}
\bigg(\frac{\sum_{i=1}^{n+1} \bar{H}_{iK} \cdot \delta_{S^y_i}}
{\sum_{i=1}^{n+1}\bar{H}_{iK}} \bigg) \bigg\}.  
\end{equation}
In other words, this is the same as the prediction set for the RLCP method
\eqref{eqn:rlcp} except with $X_K$ in place of $\tilde{X}_{n+1}$.  

Like the RLCP method, this modified construction again offers marginal coverage
under exchangeability. (Approximate conditional coverage results can also be
established, analogous to the results of Theorem \ref{thm:rlcp} for RLCP, but
for brevity we omit these here.) 

\begin{theorem}
\label{thm:rlcp_mod}
If $Z_1,\dots,Z_{n+1}$ are exchangeable, then the modified RLCP set
in \eqref{eqn:rlcp_mod} satisfies a marginal coverage guarantee, 
\[
\PP{Y_{n+1}\in\Ccal(X_{n+1})}\geq 1-\alpha.
\]
\end{theorem}

\subsubsection{View from the unified framework}

We describe how our modified RLCP method fits into the unified framework.

\paragraph{Choices of $U$ and $Q_{Z\mid U}$.}

Analogous to RLCP in Section \ref{sec:rlcp}, we let $U = (\Zbag,
X_K)$. Now, define $Q_{Z\mid U}$ as follows: at $u=(\zbag,x)$, with 
$z=((x_1,y_1),\dots,(x_{n+1},y_{n+1}))$, 
\[
Q_{Z\mid U=(\zbag,x)} = \frac{\sum_{\sigma\in\Scal_{n+1}}
\frac{H(x_{\sigma(n+1)},x)}{\sum_{k=1}^{n+1}H(x_{\sigma(n+1)},x_k)} 
\cdot\delta_{z_\sigma}}
{\sum_{\sigma\in\Scal_{n+1}}\frac{H(x_{\sigma(n+1)},x)} 
{\sum_{k=1}^{n+1}H(x_{\sigma(n+1)},x_k)}}
\]
We again write the score function as $\score(z,u) = \score(z_{n+1},\zbag)$ for
$u=\zbag$.  

\paragraph{P-value.}

Under these choices, the p-value in \eqref{eqn:unified_pvalue} is 
\begin{align*}
p &= \frac{\sum_{\sigma\in\Scal_{n+1}} \frac{H(X_{\sigma(n+1)},X_K)}
{\sum_{k=1}^{n+1}H(X_{\sigma(n+1)},X_k)} \cdot
\One{\score(Z_{\sigma(n+1)},\Zbag) \geq \score(Z_{n+1},\Zbag)}}
{\sum_{\sigma\in\Scal_{n+1}} \frac{H(X_{\sigma(n+1)},X_K)}
{\sum_{k=1}^{n+1}H(X_{\sigma(n+1)},X_k)}} \\
&= \frac{\sum_{i=1}^{n+1} \frac{H(X_i,X_K)}
{\sum_{k=1}^{n+1}H(X_i,X_k)} \cdot \One{\score(Z_i,\Zbag) \geq 
  \score(Z_{n+1},\Zbag)}}{\sum_{i=1}^{n+1} \frac{H(X_i,X_K)}
{\sum_{k=1}^{n+1}H(X_i,X_k)}} \\
&= \frac{\sum_{i=1}^{n+1} \bar{H}_{iK} \cdot \One{\score(Z_i,\Zbag) 
  \geq \score(Z_{n+1},\Zbag)}}{\sum_{i=1}^{n+1} \bar{H}_{iK}}.
\end{align*}
similarly to the RLCP calculations. And again as for RLCP, by Lemma
\ref{lem:pvalue_vs_quantile}, this corresponds to the modified RLCP set defined
in \eqref{eqn:rlcp_mod}, since  
\[
p > \alpha \iff \score(Z_{n+1},\Zbag) \leq \Quantile_{1-\alpha}\bigg(
\frac{\sum_{i=1}^{n+1} \bar{H}_{iK} \cdot
  \delta_{\score(Z_i,\Zbag)}}{\sum_{i=1}^{n+1} \bar{H}_{iK}} \bigg).  
\]

\paragraph{Validity.}

We prove Theorem \ref{thm:rlcp_mod} using the unified result in Theorem
\ref{thm:unified}. 

\begin{proof}[Proof of Theorem \ref{thm:rlcp_mod} via the unified framework.]
Fix any $z=((x_1,y_1),\dots,(x_{n+1},y_{n+1}))$. Then by definition of the
distribution of the random index $K$, we have  
\[
U\mid Z=z \sim \frac{\sum_{k=1}^{n+1} H(x_{n+1},x_k) \cdot 
\delta_{(\zbag,x_k)}} {\sum_{k=1}^{n+1}H(x_{n+1},x_k)}.
\] 
Similarly, for any permutation $\sigma\in\Scal_{n+1}$, we can calculate (using
$\zbag = \lbag z_\sigma\rbag$) that 
\[
U\mid Z=z_\sigma \sim \frac{\sum_{k=1}^{n+1} H(x_{\sigma(n+1)},x_k) \cdot
\delta_{(\zbag,x_k)}} {\sum_{k=1}^{n+1}H(x_{\sigma(n+1)},x_k)}.
\]
Comparing these two calculations, since $Z$ is assumed to be exchangeable, we
therefore have 
\[
P_{Z\mid U=(\zbag,x_i)} = \frac{\sum_{\sigma\in\Scal_{n+1}} 
\frac{H(x_{\sigma(n+1)},x_i)}{\sum_{k=1}^{n+1}H(x_{\sigma(n+1)},x_k)}
\cdot\delta_{z_\sigma}} {\sum_{\sigma\in\Scal_{n+1}}
\frac{H(x_{\sigma(n+1)},x_i)}{\sum_{k=1}^{n+1}H(x_{\sigma(n+1)},x_k)}}.
\] 
We therefore see that $P_{Z\mid U}=Q_{Z\mid U}$, which proves that $\PP{p\leq
  \alpha}\leq\alpha$ (recall Remark \ref{rmk:exact_validity}). 
\end{proof}

\subsection{Generalized WCP with a nonsymmetric score}
\label{sec:gwcp_ns}

We revisit the generalized WCP setting from Section \ref{sec:gwcp}. Recall here
the dataset $Z \in \Zcal^{n+1}$ has an arbitrary joint density $f$, and the
generalized WCP set \eqref{eqn:gwcp} is constructed by placing weight
\smash{$W^y_i = \sum_{\sigma : \sigma(n+1)=i} f(Z^y_\sigma) / f(Z^y)$} on point
$i$. The previous treatment assumed the score function is of the form
$\score(z,u) = \score(z_{n+1},\zbag)$, i.e., the score assigned to $z_{n+1}$ is
not allowed to depend on the ordering of the remaining $n$ points in
$z$. However, in this subsection we will now see that such a restriction is
actually unnecessary, from the perspective of the unified framework. Relaxing
it, as we will do below, leads to a generalized WCP set that allows for
\emph{any} score function, without a symmetry assumption. 

\subsubsection{Method and theory}

The score function now takes the form $\score((x,y),z)$, comparing a data point
$(x,y)$ to an \emph{ordered} dataset $z\in\Zcal^{n+1}$ (note that $\score$ is no
longer required to be symmetric in its second argument). This accommodates, e.g.,
a score that can depend on the time-ordering of data collected in the FCS
application (with likelihood ratio in \eqref{eqn:fcs_lr}). For arbitrary $y \in 
\Ycal$, we define as usual $Z^y = (Z_1,\dots,Z_n,(X_{n+1},y))$, and scores that
are now indexed by permutations,     
\[
S^y_\sigma = \score\big( Z^y_{\sigma(n+1)}, Z^y_\sigma \big), \quad
\sigma\in\Scal_{n+1}.  
\]
The generalized WCP set at coverage level $1-\alpha$ is defined by
\begin{equation}
\label{eqn:gwcp_ns}
\Ccal(X_{n+1}) = \bigg\{ y\in\Ycal : S^y_\id \leq \Quantile_{1-\alpha}\bigg(
\frac{\sum_{\sigma \in \Scal_{n+1}} f(Z^y_\sigma) / f(Z^y) \cdot
  \delta_{S^y_\sigma}} {\sum_{\sigma \in \Scal_{n+1}} f(Z^y_\sigma) / f(Z^y)}
\bigg)\bigg\}, 
\end{equation}
where $\id$ denotes the identity permutation. This construction with a
nonsymmetric score has the same coverage guarantee, assuming exact knowledge 
of the density ratio. 

\begin{theorem}
\label{thm:gwcp_ns}
If $Z\in\Zcal^{n+1}$ has density $f$ with respect to an exchangeable base
measure, then for any score function (not necessarily symmetric), the
generalized WCP set in \eqref{eqn:gwcp_ns} satisfies $\PP{Y_{n+1} \in
  \Ccal(X_{n+1})} \geq 1-\alpha$.  
\end{theorem}

\subsubsection{View from the unified framework}

We describe how this extension of generalized WCP fits into the unified framework.

\paragraph{Choices of $U$ and $Q_{Z\mid U}$.}

As for generalized WCP in Section \ref{sec:gwcp}, we define $U = \Zbag$, but we
now write the score function as 
\[
\score(z,u) = \score(z_{n+1},z).
\]
(Compare this to $\score(z_{n+1},\zbag)$ for the symmetric setting of Section
\ref{sec:gwcp}). Define  
\[
Q_{Z\mid U=\zbag} = \frac{\sum_{\sigma\in\Scal_{n+1}} f(z_\sigma) \cdot 
    \delta_{z_\sigma}}{\sum_{\sigma\in\Scal_{n+1}} f(z_\sigma)}
= \frac{\sum_{\sigma\in\Scal_{n+1}} f(z_\sigma) / f(z) \cdot
    \delta_{z_\sigma}}{\sum_{\sigma\in\Scal_{n+1}} f(z_\sigma) / f(z)}.
\]

\paragraph{P-value.}

Under these choices, the p-value in \eqref{eqn:unified_pvalue} is 
\[
p = \frac{\sum_{\sigma\in\Scal_{n+1}} f(Z_\sigma) / f(Z) \cdot
  \One{\score(Z_{\sigma(n+1)},Z_\sigma) \geq \score(Z_{n+1},Z)}}
  {\sum_{\sigma\in\Scal_{n+1}} f(Z_\sigma) / f(Z)} .
\]
To see that this corresponds to generalized WCP, recall that by Lemma 
\ref{lem:pvalue_vs_quantile},  
\[
p>\alpha \iff \score(Z_{n+1},Z) \leq \Quantile_{1-\alpha}\bigg(
\frac{\sum_{\sigma\in\Scal_{n+1}} f(Z_\sigma) / f(Z) \cdot
  \delta_{\score(Z_{\sigma(n+1)},Z_\sigma)}}
  {\sum_{\sigma\in\Scal_{n+1}} f(Z_\sigma) / f(Z)}\bigg). 
\]  
The right-hand side above can be directly seen to be equivalent to the event
$Y_{n+1} \in \Ccal(X_{n+1})$, for the generalized WCP set in \eqref{eqn:gwcp_ns}. 

\paragraph{Validity.}

We prove Theorem \ref{thm:gwcp_ns} using the unified result in Theorem
\ref{thm:unified}. 

\begin{proof}[Proof of Theorem \ref{thm:gwcp_ns} via the unified framework] 
The proof is identical to that of Theorem \ref{thm:gwcp}: since we have assumed
that $Z$ has joint density $f$ with respect to some exchangeable base measure,
this means that $P_{Z\mid U} = Q_{Z\mid U}$, by construction, and therefore
(recalling Remark \ref{rmk:exact_validity}) we have $\PP{p\leq \alpha}\leq
\alpha$.   
\end{proof}

\subsection{Generalized WCP with Monte Carlo p-values}

In the generalized WCP setting once again, an important bottleneck that can
arise in practice is that the generalized WCP set \eqref{eqn:gwcp} can be very
expensive to compute, due to the calculation required for each weight $W^y_i$:
indeed, each $W^y_i$ here is a sum over $n!$ terms, where each term requires
computing the likelihood ratio $f(Z^y_\sigma) / f(Z^y)$, which can itself be
highly nontrivial. Of course, computational tractability is only made worse in
the extension \eqref{eqn:gwcp_ns} to nonsymmetric scores given in the last
subsection. We present in what follows a Monte Carlo approximation to this
generalized WCP set, which will improve computational efficiency while
preserving the same validity guarantee as the original construction.  

In this subsection, we will specialize to a setting where the covariates have a
known but potentially complex dependence: 
\[
X = (X_1,\dots,X_{n+1}) \sim P_X,
\]
while $Y_i \mid X$, for $i \in [n+1]$, are independent from a common conditional
distribution $P_{Y\mid X}$. If $f$ denotes the joint density of $Z$ and $f_X$
the corresponding density of $X$, then observe that under this model, for any $z
= ((x_1,y_1),\dots,(x_{n+1},y_{n+1}))$, 
\begin{equation}\label{eqn:fx}
\frac{f(z_\sigma)}{f(z)} =
\frac{f_X(x_{\sigma(1)},\dots,x_{\sigma(n+1)})}{f_X(x_1,\dots,x_{n+1})}, 
\end{equation}
where this simplification is due to our assumption on the distribution of
$(Y_1,\dots,Y_{n+1})\mid X$. Critically, the likelihood ratio expression on the  
right-hand side above is only a function of the features. 

\subsubsection{Method and theory}

For arbitrary $y \in \Ycal$, we define $Z^y = (Z_1,\dots,Z_n,(X_{n+1},y))$ as
usual, and just as in the last subsection, define scores that are indexed by
permutations,  
\[
S^y_\sigma = \score\big( Z^y_{\sigma(n+1)}, Z^y_\sigma \big), \quad
\sigma\in\Scal_{n+1}.
\]

Fixing $M \geq 1$, let $g_X$ be a user-chosen proposal density, where $f_X,g_X$
are assumed to be absolutely continuous with respect to one another. Conditional
on $\lbag X \rbag$, we can think of $g_X$ as inducing a distribution on
permutations, which places probability  
\[
\frac{g_X(X_\sigma)}{\sum_{\sigma' \in \Scal_{n+1}} g_X(X_{\sigma'})}
\]
on $\sigma \in \Scal_{n+1}$. Let $\sigma_1,\dots,\sigma_M$ be i.i.d.\ samples
from this distribution, and set $\sigma_0 = \id$ to be the identity permutation,
for convenience. Then the importance sampling approximation to the generalized
WCP set at coverage level $1-\alpha$ is defined by 
\begin{equation}
\label{eqn:gwcp_ns_is}
\Ccal(X_{n+1}) = \bigg\{ y\in\Ycal : S^y_\id \leq \Quantile_{1-\alpha}\bigg(
\frac{\sum_{m=0}^M \frac{f_X(X_{\sigma_m}) / f_X(X)}{g_X(X_{\sigma_m}) / g_X(X)} \cdot
  \delta_{S^y_{\sigma_m}}} {\sum_{m=0}^M \frac{f_X(X_{\sigma_m}) /
    f_X(X)}{g_X(X_{\sigma_m}) / g_X(X)}} \bigg)\bigg\}.  
\end{equation} 
This Monte Carlo set, while computationally cheaper, still has an exact validity 
guarantee. 

\begin{theorem}
\label{thm:gwcp_ns_is}
Suppose $Z\in\Zcal^{n+1}$ has density $f$ with respect to an exchangeable base
measure, with $f$ satisfying \eqref{eqn:fx}. Let $g_X$ be an arbitrary proposal
density (such that $f_X,g_X$ are absolutely continuous with respect to each
other). Then for any choice $M \geq 1$, the importance sampling approximation to
the generalized WCP set in \eqref{eqn:gwcp_ns_is} satisfies $\PP{Y_{n+1} 
 \in \Ccal(X_{n+1})} \geq 1-\alpha$. 
\end{theorem}

\subsubsection{View from the unified framework}

We describe how this importance sampling version of generalized WCP fits into  
the unified framework.

\paragraph{Choices of $U$, $Q_{Z \mid U}$, and $R_{Z \mid U}$.} 

We define $U$ and $Q_{Z\mid U}$ precisely as in generalized WCP for a
nonsymmetric score, as in Section \ref{sec:gwcp_ns}. Moreover, we take $R_{Z
  \mid U}$ to be the conditional distribution corresponding to joint density
$g_X$ on $X=(X_1,\dots,X_{n+1})$, namely, for $u=\zbag$ where $z =
((x_1,y_1),\dots,(x_{n+1},y_{n+1}))$, 
\[
R_{Z\mid U=\zbag} = \frac{\sum_{\sigma\in\Scal_{n+1}}
  g_X(x_{\sigma(1)},\dots,x_{\sigma(n+1)}) \cdot \delta_{z_\sigma}}
{\sum_{\sigma\in\Scal_{n+1}} g_X(x_{\sigma(1)},\dots,x_{\sigma(n+1)})}.
\]

\paragraph{P-value.}

Under these choices, the importance sampling p-value in
\eqref{eqn:unified_pvalue_is} is 
\[
\tilde{p} = \frac{\sum_{m=0}^M \frac{\d{Q}_{Z \mid U}(Z_{\sigma_m})}
{\d{R}_{Z \mid U}(Z_{\sigma_m})} \cdot 
\One{\score(Z_{\sigma_m(n+1)},Z_{\sigma_m}) \geq \score(Z_{n+1},Z)}}
{\sum_{m=0}^M \frac{\d{Q}_{Z \mid U}(Z_{\sigma_m})}
{\d{R}_{Z \mid U}(Z_{\sigma_m})}},
\]
where we are writing $Z^{(m)}=Z_{\sigma_m}$, for some permutation
$\sigma_m$. Plugging in our choices of $Q_{Z\mid U}$ and $R_{Z\mid U}$, then,  
\[
\tilde{p} =\frac{\sum_{m=0}^M \frac{f_X(X_{\sigma_m})/f_X(X)}
{g_X(X_{\sigma_m})/g_X(X)} \cdot \One{\score(Z_{\sigma_m(n+1)},Z_{\sigma_m}) 
\geq \score(Z_{n+1},Z)}}{\sum_{m=0}^M \frac{f_X(X_{\sigma_m})/f_X(X)}
{g_X(X_{\sigma_m})/g_X(X)}} .
\]
To see that this corresponds to generalized WCP, observe that by Lemma 
\ref{lem:pvalue_vs_quantile},  
\[
\tilde{p}>\alpha \iff \score(Z_{n+1},Z) \leq \Quantile_{1-\alpha}\bigg(
\frac{\sum_{m=0}^M \frac{f_X(X_{\sigma_m}) / f_X(X)}{g_X(X_{\sigma_m})/g_X(X)} 
  \cdot \delta_{\score(Z_{\sigma_m(n+1)},Z_{\sigma_m})}}
{\sum_{m=0}^M \frac{f_X(X_{\sigma_m}) / f_X(X)}{g_X(X_{\sigma_m})/g_X(X)}} 
\bigg). 
\]  
The right-hand side above can be directly seen to be equivalent to the event
$Y_{n+1} \in \Ccal(X_{n+1})$, for the importance sampling generalized WCP set in
\eqref{eqn:gwcp_ns_is}.  

\paragraph{Validity.}

We prove Theorem \ref{thm:gwcp_ns_is} using the unified result in Theorem
\ref{thm:unified_is}.  

\begin{proof}[Proof of Theorem \ref{thm:gwcp_ns_is} via the unified framework]  
The proof is similar to that of Theorem \ref{thm:gwcp_ns}: since we have assumed
that $Z$ has joint density $f$ with respect to some exchangeable base measure,
this means that $P_{Z\mid U} = Q_{Z\mid U}$, by construction, and thus (now
using Remark \ref{rmk:exact_validity_is}) we have $\PP{p\leq \alpha}\leq
\alpha$.     
\end{proof}

\section{Discussion}

In this work, we presented a unified treatment of different existing methods 
and theorems in the conformal prediction literature which handle departures from 
the standard assumption of exchangeability. At its core, our unified view of 
conformal methods is that they are based on inference about the joint
distribution of data $Z=(Z_1,\dots,Z_{n+1})$ (where each $Z_i=(X_i,Y_i)$), given 
partial information $U$ about the data. The partial information $U$ contains at
least the information in $\Zbag$ (the unordered bag of points in $Z$) and it may
contain more. We find that many existing results can be framed as follows: the
user specifies a conditional distribution $Q_{Z \mid U}$ as a model for the true
distribution of $Z \mid U$, and then uses $Q_{Z \mid U}$ to form a prediction
set. Our theory provides a coverage guarantee that depends on the total
variation distance between $Q_{Z \mid U}$ and the true distribution $P_{Z \mid
U}$.

Apart from standard and split conformal, other CP methods encompassed by our
framework include weighted conformal prediction (WCP), nonexchangeable conformal
prediction (NexCP), randomly-localized conformal prediction (RLCP), and
generalized WCP. While we did not study hierarchical CP
\citep{lee2023distribution} or SymmPI \citep{dobriban2023symmPI}, which rely on
notions of invariance, we believe our framework should be able to accommodate,
at least to some extent, these methods as special cases. Moreover, while our
work focused on single-point prediction problems (with a single test point), we
believe our framework can be extended to cover multi-point settings, for
example, transductive prediction \citep{vovk2013transductive} or selective
prediction \citep{jin2023selective, jin2023model}.

We also showed that new conformal prediction results follow from the unified
framework. Many others should also be possible to derive. For example, two
existing methods, WCP and NexCP, are both weighted variants of conformal
prediction but they appear quite different in their assumptions and technical
motivations. Our paper provides a unified explanation for why they both work;
this suggests the possibility of combining the two styles of weights, and raises
an open question of determining problem settings in which such a combination
might be useful. As another example, the results we presented for generalized
WCP assume that the likelihood ratio between different possible data
permutations is known exactly; robustness results (such as those for WCP) would
be important to develop, for applications where this likelihood ratio is unknown
or cannot be computed exactly due to computational complexity. The unified
framework not only offers a route towards developing these possible extensions,
but may also enable the development of completely new methods, for settings not
currently covered by the footprint of existing conformal methodology.   

We conclude by discussing how our paper fits into the broader landscape of
conditional inference in statistics. Conditioning is of course a key device in
statistics, and is ubiquitous in both classical and modern inferential
principles and tools. Conditioning lies at the core of Fisherian inference,
featured prominently in core ideas such as Fisher's sufficiency principle
\citep{fisher1922mathematical}, Birnbaum's conditionality principle
\citep{birnbaum1962foundations}, and the Rao-Blackwell theorem
\citep{rao1945information, blackwell1947conditional}. Its importance can also be
clearly understood from the classic books by \citet{lehmann1998theory,
lehmann2022testing}. Conditioning plays a similarly major role in nonparametric
inference: permutation and randomization tests being two key examples, each
based on conditioning. These bear strong similarities, but are motivated from
different perspectives. A recent paper by \citet{zhang2023what} provides a nice
overview, and interprets CP and WCP in terms of quasi-randomization tests. 

More broadly, conditioning lies at the center of many developments in modern
statistical inference. This includes selective inference \citep{lee2016exact,
tibshirani2016exact, fithian2014optimal}, adaptive data analysis
\citep{dwork2015preserving, dwork2015reusable, bassily2016algorithmic},
conditional independence testing \citep{candes2018panning,
berrett2020conditional}, and joint coverage regions \citep{dobriban2023joint}.  
Our paper contributes to the literature on conditional inference, revealing that
conditioning on partial information $U$ is one of the main ideas underpinning
conformal prediction and its many extensions beyond exhangeability, with
other one being the choice of a conditional distribution $Q_{Z\mid U}$. The
connection to testing is more salient, and the validity and robustness results
associated with conformal prediction and generalizations should now be less
mysterious when viewed from the traditional statistical lens. Connections to
seemingly disparate parts of the literature on conditional inference may be
within reach, and may represent interesting directions for future work.  

\subsection*{Acknowledgements}

We are grateful to our colleagues Emmanuel Cand{\`e}s, Aaditya Ramdas,
Anastasios Angelopoulos, and Stephen Bates for many inspiring collaborations and
discussions which have helped shape our understanding of conformal prediction, 
and underlie the ideas in this work. We are also grateful to Drew Prinster for
providing helpful commentary and references. 

R.F.B.\ was supported by National Science Foundation (NSF) grant number
DMS-2023109 and Office of Naval Research (ONR) grant number
N00014-24-1-2544. R.J.T.\ was supported by ONR grant number N00014-20-1-2787.  

{\small\RaggedRight
\bibliographystyle{plainnat}
\bibliography{bib}}

\clearpage
\appendix

\section{Additional proofs and technical results}

\subsection{Proof of Lemma \ref{lem:pvalue_vs_quantile}}
\label{app:pvalue_vs_quantile}

We can assume $\alpha<1$ since otherwise the claim is trivial (we would have
$\Quantile_{1-\alpha}(Q) = -\infty$). Let $x_1 < \dots < x_m\in\R$ be the  
values in the support of $Q$, and $p_i = \Pp{Q}{X=x_i}>0$, for each $i\in[m]$.
Then by definition of the quantile,   
\[
\Quantile_{1-\alpha}(Q) = \inf\{t\in\R : \Pp{Q}{X\leq t}\geq 1-\alpha\} =
x_{k^*},
\]
where 
\[
k^* = \min\bigg\{ k \in [m]: \sum_{i \leq k} p_i \geq 1-\alpha \bigg\}.  
\]
Therefore, for any $x\in\R$, if $x \leq x_{k^*}$ then 
\[
\Pp{Q}{X\geq x} = \sum_{i=1}^m p_i \One{x_i \geq x} \geq \sum_{i=1}^m p_i
\One{x_i \geq x_{k^*}} = \sum_{i\geq k^*} p_i = 1 - \sum_{i<k^*}p_i >\alpha,
\]
since $\sum_{i < k^*} p_i < 1-\alpha$ by definition of $k^*$. And similarly, if
$x>x_{k^*}$ then 
\[
\Pp{Q}{X\geq x} = \sum_{i=1}^m p_i \One{x_i \geq x} \leq \sum_{i=1}^m p_i
\One{x_i > x_{k^*}} = \sum_{i> k^*} p_i = 1 - \sum_{i\leq k^*}p_i \leq\alpha,
\] 
since $\sum_{i \leq k^*} p_i \geq 1-\alpha$ by definition of $k^*$.

\subsection{Proof of Theorem \ref{thm:unified_mc}}
\label{app:proof_thm:unified_mc}

This result is a consequence of Theorem \ref{thm:unified_is} for the importance
sampling setting, obtained by simply choosing the proposal distribution as
$R_{Z\mid U}=Q_{Z\mid U}$. 

\subsection{Proof of Theorem \ref{thm:unified_is}}
\label{app:proof_thm:unified_is}

Although we have presented this result as a guarantee for $\tilde{p}$, which is
approximation of the original p-value $p$ that was defined
in \eqref{eqn:unified_pvalue} and analyzed in Theorem \ref{thm:unified}, we will
now see that Theorem \ref{thm:unified_is} can in fact be derived (exactly) as a
special case of Theorem \ref{thm:unified}. 

We will first need to reset some of our notation. Specifically, we need a new
definition of partial information: instead of $U$ on its own, the partial
information will now be given by  
\[\tilde{U} = \big(U, \lbag Z^{(0)},\dots,Z^{(M)} \hspace{-3pt}\rbag
  \hspace{-2pt}\big), 
\]
and overloading notation, we will let $\score(Z,\tilde{U})=\score(Z,U)$. We will
also define the conditional distribution 
\[
Q_{Z\mid \tilde{U} = (u,\lbag z^{(0)},\dots,z^{(M)}\rbag)} = \frac{\sum_{m=0}^M
\frac{\d{Q_{Z\mid U=u}}}{\d{R_{Z\mid U=u}}}(z^{(m)})\cdot \delta_{z^{(m)}}}
{\sum_{m=0}^M\frac{\d{Q_{Z\mid U=u}}}{\d{R_{Z\mid U=u}}}(z^{(m)})}.
\]
Observe that the p-value $\tilde{p}$ defined in \eqref{eqn:unified_pvalue_is} is
exactly the same as the original p-value $p$ in \eqref{eqn:unified_pvalue} if we
use the new conditional distribution $Q_{Z\mid \tilde{U}}$ in place of the
original one $Q_{Z\mid U}$. Applying Theorem \ref{thm:unified} (together with
Remark \ref{rmk:information_monotonicity}), we therefore have the guarantee 
\[
\PP{\tilde{p}\leq\alpha} \leq \alpha + \inf_{Q_{\tilde{U}}} \,
\dtv(P_{Z,\tilde{U}},Q_{Z,\tilde{U}}), 
\]
where the infimum is taken over all distributions $Q_{\tilde{U}}$ on $\tilde{U}
= (U, \lbag Z^{(0)},\dots,Z^{(M)} \rbag)$, and where $Q_{Z,\tilde{U}} =
Q_{Z\mid\tilde{U}}\times Q_{\tilde{U}}$. 

Now we need to work with this remaining TV term. First, we can observe that
$P_{Z,\tilde{U}}$ is the distribution on $(Z,\tilde{U}) = (Z,U,\lbag Z,
Z^{(1)},\dots,Z^{(M)}\rbag)$ induced by the joint distribution 
\[
(Z,U,Z^{(1)},\dots,Z^{(M)})\sim P_{Z,U}\times (Q_{Z\mid U})^M,
\] 
where this notation indicates that we first draw $(Z,U)\sim P_{Z,U}$, and then 
conditional on this draw, we sample \smash{$Z^{(1)},\dots,Z^{(M)}\iidsim
  Q_{Z\mid U}$}. Next, fix any marginal $Q_U$ on $U\in\Ucal$, sample 
\[
(Z,U,Z^{(1)},\dots,Z^{(M)})\sim Q_{Z,U}\times (Q_{Z\mid U})^M,
\]
and define $Q_{\tilde{U}}$ as the corresponding distribution of
$\tilde{U}=(U,\lbag Z,Z^{(1)},\dots,Z^{(M)}\rbag)$. Note that by exchangeability
of $Z,Z^{(1)},\dots,Z^{(M)}$ under this joint distribution, the induced
distribution of $(Z,\tilde{U})$ is thus equivalent to
$Q_{Z,\tilde{U}}=Q_{Z\mid\tilde{U}} \times Q_{\tilde{U}}$. And by information 
monotonicity, we have 
\[
\dtv(P_{Z,\tilde{U}},Q_{Z,\tilde{U}}) \leq \dtv\big(P_{Z,U}\times (Q_{Z\mid
  U})^M ,Q_{Z,U}\times (Q_{Z\mid U})^M\big) = \dtv(P_{Z,U},Q_{Z,U}). 
\]
In other words, we have verified that
\[\inf_{Q_{\tilde{U}}}\dtv(P_{Z,\tilde{U}},Q_{Z,\tilde{U}})
\leq \inf_{Q_U} \dtv(P_{Z,U},Q_{Z,U}),
\]
which completes the proof.

\subsection{Refinement of Theorem \ref{thm:unified_mc}}
\label{app:unified_mc_refine}

\begin{theorem}
\label{thm:unified_mc_refine}
Suppose $(Z,U) \sim P_{Z,U}$. Fix any $M\geq 1$ with $\alpha(M+1)\geq 1$. 
Let \smash{$P_{S,\tilde{T}}$} be the induced joint distribution on
\smash{$(S,\tilde{T})=(\score(Z,U),\thresh_A(U))$}, where 
\smash{$A\sim \textnormal{Beta}_{M,\alpha}$} is independent of $(Z,U)$, for 
\[
\textnormal{Beta}_{M,\alpha} = \textnormal{Beta}\Big(\lfloor
\alpha(M+1)\rfloor,\lceil (1-\alpha)(M+1) \rceil\Big). 
\] 
Then the Monte Carlo p-value defined in \eqref{eqn:unified_pvalue_mc} satisfies   
\[
\PP{\tilde{p}\leq \alpha}\leq \alpha + \inf_{Q_U} \, \dtv(P_{S,\tilde{T}},Q_{S,\tilde{T}}),
\]
where the infimum is taken over all distributions $Q_U$ on $U$, and where
\smash{$Q_{S,\tilde{T}}$} denotes the joint distribution of
\smash{$(S,\tilde{T})=(\score(Z,U),\thresh_A(U))$} induced by drawing
\smash{$(Z,U,A) \sim Q_{Z,U}\times \textnormal{Beta}_{M,\alpha}$}, for $Q_{Z,U}
= Q_{Z\mid U}\times Q_U$.  
\end{theorem}

\begin{proof}
Let $Q_{S\mid U=u}$ denote the induced distribution on $S=\score(Z,u)$, under
$Z\sim Q_{Z\mid U=u}$. Since $S\in\R$, we can generate samples from this
distribution by using its quantile function---that is, if
$V\sim\textnormal{Unif}[0,1]$, then defining  
\[
S = \Quantile_V(Q_{S\mid U=u}),
\]
this generates a random draw from $Q_{S\mid U=u}$.

Now let $V_1,\dots,V_M\iidsim \textnormal{Unif}[0,1]$, independent of
$(Z,U)$. Defining $S_m =\Quantile_{V_m}(Q_{S\mid U})$ for each $m\in[M]$, we can
therefore equivalently define the p-value $\tilde{p}$ as  
\[
\tilde{p} = \frac{1 + \sum_{m=1}^M\One{S_m\geq \score(Z,U)}}{M+1}.
\]
By construction, we can observe that, for $k =
\lceil(1-\alpha)(M+1)\rceil\in[M]$,  
\[
\tilde{p} \leq \alpha \iff \score(Z,U) >  S_{(k)},
\]
where $S_{(1)}\leq \dots \leq S_{(M)}$ are the order statistics of $S_1,\dots,S_M$.
Since $v\mapsto \Quantile_v(Q_{S\mid U})$ is a monotone nondecreasing
function, it holds that $S_{(k)} = \Quantile_{V_{(k)}}(Q_{S\mid U})$. Therefore, 
\[
\tilde{p} \leq \alpha \iff \score(Z,U) >  \Quantile_{V_{(k)}}(Q_{S\mid U}) =
\thresh_{1-V_{(k)}}(U), 
\]
where the last step uses the definition of the threshold $t_a(u) =
\Quantile_{1-a}(Q_{S\mid U=u})$. Further, by construction, we have
$V_{(k)}\independent (Z,U)$, and $1-V_{(k)} \sim \textnormal{Beta}(M+1-k,k) =
\textnormal{Beta}_{M,\alpha}$.   

In other words, setting $\tilde{T} = \thresh_A(U)$ where we define
$A=1-V_{(k)}\sim \textnormal{Beta}_{M,\alpha}$, so far we we have shown that 
\[\
Pp{P_{Z,U}\times (Q_{Z\mid U})^M}{\tilde{p}\leq \alpha} = 
\Pp{P_{S,\tilde{T}}}{S > \tilde{T}}.
\]
But by an identical argument, it also holds that
\[
\Pp{Q_{Z,U}\times (Q_{Z\mid U})^M}{\tilde{p}\leq \alpha} =
\Pp{Q_{S,\tilde{T}}}{S > \tilde{T}}.
\]
We therefore have
\begin{align*}
\Pp{P_{Z,U}\times (Q_{Z\mid U})^M}{\tilde{p}\leq \alpha}
&=\Pp{P_{S,\tilde{T}}}{S> \tilde{T}} \\
&\leq\Pp{Q_{S,\tilde{T}}}{S>\tilde{T}} +\dtv(P_{S,\tilde{T}},Q_{S,\tilde{T}}) \\
&=\Pp{Q_{Z,U}\times (Q_{Z\mid U})^M}{\tilde{p}\leq \alpha} +
  \dtv(P_{S,\tilde{T}},Q_{S,\tilde{T}}) \\
&\leq \alpha + \dtv(P_{S,\tilde{T}},Q_{S,\tilde{T}}),
\end{align*}
where the last step holds since, under the joint distribution $Q_{Z,U}\times
(Q_{Z\mid U})^M$, by construction it holds that $Z,Z^{(1)},\dots,Z^{(M)}$ are
exchangeable conditional on $U$ (because, conditional on $U$, these random   
variables comprise $M+1$ i.i.d.\ draws from $Q_{Z\mid U}$), and hence
$\tilde{p}$ is superuniform under this distribution. 
\end{proof}

\subsection{Proof of Theorem \ref{thm:unified_un}}
\label{app:unified_un}

We will need two additional supporting lemmas. The first is due to Lemma A1 from
\citet{harrison2012conservative}, and we transcribe it to our notation here for
concreteness.  

\begin{lemma}[{\citealt{harrison2012conservative}}]
\label{lem:pvalue_measure}
Let $Q$ be a measure on a space $\Zcal$, which is supported on finitely many
values. Let $\teststat:\Zcal\to\R$ be a measurable function. Define\footnote{For
  conciseness, throughout Appendix \ref{app:unified_un}, for the unnormalized
  setting we are using capital letters, such as $Z$, to define events---for
  example, $\{p(Z)\leq \alpha\}$ should be interpreted as the subset
  $\{z\in\Zcal : p(z)\leq \alpha\}\subseteq\Zcal$, and consequently, for a
  measure $Q$ on $\Zcal$, the notation $Q\{p(Z)\leq \alpha\}$ should be
  interpreted as $Q\big(\{z\in\Zcal : p(z)\leq\alpha\}\big)$.} 
\[
p(z) = Q\{\teststat(Z) \geq \teststat(z)\}. 
\]
Then for all $\alpha \geq 0$,
\[
Q\{p(Z) \leq \alpha\} \leq \alpha.
\]
\end{lemma}

When $Q$ is a distribution (i.e., $Q(\Zcal) =1$) this lemma reduces to the
standard fact in \eqref{eqn:pvalue_superuniform}: a p-value constructed for
distribution $Q$ is valid (i.e., superuniform) for data drawn from $Q$. 
The next lemma is similar to Lemma \ref{lem:pvalue_vs_quantile}, and its proof
is similar to the proof of this lemma in Appendix \ref{app:pvalue_vs_quantile},
hence omitted.  

\begin{lemma}
\label{lem:p_vs_q_measure}
Let $Q$ be a measure on $\R$, which is supported on finitely many values. Then
for any $\alpha\geq 0$ and $x\in\R$,  
\[ 
Q\{X\geq x\} > \alpha \iff x \leq \inf\{t\in\R : Q\{X>t\} \leq \alpha\}.  
\]
\end{lemma}

When $Q$ is a distribution this lemma reduces to the result in Lemma
\ref{lem:pvalue_vs_quantile}. We now give the proof of Theorem
\ref{thm:unified_un}. 

\begin{proof}[Proof of Theorem \ref{thm:unified_un}]
By Lemma \ref{lem:pvalue_measure} (applied with $Q_{Z \mid U=u}$ in place of
$Q$, and the score $\score(\cdot,u)$ in place of the test statistic
$\teststat$), we have   
\[
Q_{Z \mid U=u}\{p(Z,u) \leq \alpha\} \leq \alpha.
\]
Since this holds for each $u\in\Ucal$, we thus have for any marginal
distribution $Q_U$,
\[
\Ep{Q_U}{Q_{Z \mid U}\{p(Z,U) \leq \alpha\}} \leq \alpha. 
\]
Next, let $Q_{S\mid U=u}$ denote the measure on $\score(Z,u)$ induced by the
measure $Q_{Z\mid U=u}$ on $\Zcal$. Then by Lemma \ref{lem:p_vs_q_measure}
(applied with $Q_{S\mid U=u}$ in place of $Q$), we have for any $x\in\R$, 
\[
Q_{S\mid U=u}\{S\geq x\} >\alpha \iff x \leq \inf\big\{t\in\R : Q_{S\mid
  U=u}\{S>t\} \leq \alpha\big\}. 
\]
Equivalently, plugging in the definition of $Q_{S\mid U=u}$ and
$\thresh_\alpha(u)$, we can write 
\[
Q_{Z\mid U=u}\{\score(Z,u)\geq x\} >\alpha \iff x \leq \thresh_\alpha(u).
\]
Fixing any $z\in\Zcal$ and plugging in $x=\score(z,u)$, we therefore have
\[
Q_{Z\mid U=u}\{\score(Z,u)\geq \score(z,u)\} >\alpha \iff \score(z,u) \leq
\thresh_\alpha(u). 
\]
But since $p(z,u) = Q_{Z\mid U=u}\{\score(Z,u)\geq \score(z,u)\}$ by definition
of the p-value in the unnormalized setting, given in
\eqref{eqn:unified_pvalue_un}, we equivalently have   
\[
p(z,u) >\alpha \iff \score(z,u) \leq \thresh_\alpha(u).
\]
Therefore,
\[
\Ep{Q_U}{Q_{Z \mid U}\{\score(Z,U) > \thresh_\alpha(U)\}} = \Ep{Q_U}{Q_{Z \mid
    U}\{p(Z,U) \leq \alpha\}} \leq \alpha,
\]
and hence 
\begin{align*}
\Pp{P_{Z,U}}{p(Z,U) \leq \alpha}
&=\Pp{P_{Z,U}}{\score(Z,U) > \thresh_\alpha(U)} \\
&\leq \alpha + \Pp{P_{Z,U}}{\score(Z,U) > \thresh_\alpha(U)} - 
\Ep{Q_U}{Q_{Z \mid U}\{\score(Z,U) > \thresh_\alpha(U)\}},
\end{align*}
The proof is complete by noting that the difference on the right-hand side
is bounded by the supremum in the theorem statement (taking $A = \{(z,u) 
\in \Zcal \times \Ucal : \score(z,u) > \thresh_\alpha(u)\}$).  
\end{proof}

\subsection{Additional calculations for the proof of Theorem
  \ref{thm:wcp_drift}} 
\label{app:wcp_drift}

First we verify the bound \eqref{eqn:wcp_drift_term_i} on Term $i$. By
definition of $P$, we calculate 
\[
P_{Z_{n+1},U}(A) = \Pp{P_{Z,U}}{(Z_{n+1},U)\in A} =
\Pp{\Ptr_1\times\dots\times\Ptr_n\times\Pte}{(Z_{n+1},\Zbag)\in A}.
\]
We can bound this last expression as
\begin{align*}
&\Pp{\Ptr_1\times\dots\times\Ptr_n\times\Pte}{(Z_{n+1},\Zbag)\in A} \\
&\leq \Pp{\Ptr_1\times\dots\times\Ptr_{i-1}\times\bar{\Ptr}\times
\Ptr_{i+1}\times\dots\times\Ptr_n\times\Pte}{(Z_{n+1},\Zbag)\in A} 
+\dtv(\Ptr_i,\bar{\Ptr}) \\
&= \Pp{\Ptr_1\times\dots\times\Ptr_{i-1}\times\Pte\times\Ptr_{i+1}
\times\dots\times\Ptr_n\times\bar{\Ptr}}{(Z_i,\Zbag)\in A} 
+\dtv(\Ptr_i,\bar{\Ptr}) \\
&= \Ep{\Ptr_1\times\dots\times\Ptr_{i-1}\times\Ptr_i\times\Ptr_{i+1}
\times\dots\times\Ptr_n\times\bar{\Ptr}}{w^*_i(Z_i)\cdot 
\One{(Z_i,\Zbag)\in A}}+\dtv(\Ptr_i,\bar{\Ptr}),
\end{align*}   
where the second step holds by permuting the variables, and the last step holds
by definition of $w^*_i$. Returning to the definition of Term $i$, we then
calculate  
\begin{align*}
\textnormal{Term $i$} 
&=\sup_A\,\Big\{P_{Z_{n+1},U}(A) -\Ep{\Ptr_1\times\dots\times 
\Ptr_n\times \bar{\Ptr}}{\bar{w}(Z_i)\cdot \One{(Z_i,\Zbag)\in A}}\Big\} \\
&\leq \sup_A\,\bigg\{\Big(\Ep{\Ptr_1\times\dots\times\Ptr_n\times\bar{\Ptr}} 
{w^*_i(Z_i)\cdot \One{(Z_i,\Zbag)\in A}} +\dtv(\Ptr_i,\bar{\Ptr})\Big) \\
&\hspace{2in}{}- \Ep{\Ptr_1\times\dots\times \Ptr_n\times \bar{\Ptr}}
{\bar{w}(Z_i)\cdot \One{(Z_i,\Zbag)\in A}}\bigg\} \\ 
&\leq \Ep{\Ptr_1\times\dots\times \Ptr_n\times \bar{\Ptr}}
{\big(w^*_i(Z_i)-\bar{w}(Z_i)\big)_+} + \dtv(\Ptr_i,\bar{\Ptr}) \\ 
&=\Ep{\Ptr_i}{\big(w^*_i(X,Y)-\bar{w}(X,Y)\big)_+} +\dtv(\Ptr_i,\bar{\Ptr}) ,
\end{align*}
which establishes the desired bound \eqref{eqn:wcp_drift_term_i}.

Next we verify the bound \eqref{eqn:wcp_drift_term_n+1} on Term $n+1$. By
definition of $\bar{\Ptr}$, 
\begin{multline*}
\Ep{\Ptr_1\times\dots\times \Ptr_n\times \bar{\Ptr}}{\bar{w}(Z_{n+1})\cdot
 \One{(Z_{n+1},\Zbag)\in A}} \\
= \frac{1}{n}\sum_{i=1}^n\Ep{\Ptr_1\times\dots\times \Ptr_n\times \Ptr_i}
{\bar{w}(Z_{n+1})\cdot \One{(Z_{n+1},\Zbag)\in A}}.
\end{multline*}
And, for any $i\in[n]$, by definition of $w^*_i$ we can calculate
\begin{multline*}
P_{Z_{n+1},U}(A)=\Pp{\Ptr_1\times\dots\times\Ptr_n\times\Pte}
{(Z_{n+1},\Zbag)\in A} \\
= \Ep{\Ptr_1\times\dots\times\Ptr_n\times\Ptr_i}{w^*_i(Z_{n+1})
  \cdot\One{(Z_{n+1},\Zbag)\in A}},
\end{multline*}
and so taking an average,
\[
P_{Z_{n+1},U}(A) = \frac{1}{n}\sum_{i=1}^n\Ep{\Ptr_1\times\dots\times
\Ptr_n\times\Ptr_i}{w^*_i(Z_{n+1}) \cdot\One{(Z_{n+1},\Zbag)\in A}}.
\]
Therefore, returning to the definition of Term $n+1$, we have
\begin{align*}
\textnormal{Term $n+1$}
&=\sup_A\,\Big\{P_{Z_{n+1},U}(A) - \Ep{\Ptr_1\times\dots\times \Ptr_n\times 
\bar{\Ptr}}{\bar{w}(Z_{n+1})\cdot \One{(Z_{n+1},\Zbag)\in A}}\Big\} \\
&=\sup_A\,\bigg\{ \frac{1}{n}\sum_{i=1}^n\Ep{\Ptr_1\times\dots\times
\Ptr_n\times\Ptr_i}{w^*_i(Z_{n+1}) \cdot\One{(Z_{n+1},\Zbag)\in A}} \\
&\hspace{1in}{} -\frac{1}{n}\sum_{i=1}^n\Ep{\Ptr_1\times\dots\times \Ptr_n
\times \Ptr_i}{\bar{w}(Z_{n+1})\cdot \One{(Z_{n+1},\Zbag)\in A}}\bigg\} \\
&=\sup_A\,\bigg\{ \frac{1}{n}\sum_{i=1}^n  \Ep{\Ptr_1\times\dots\times
\Ptr_n\times\Ptr_i}{\big(w^*_i(Z_{n+1})-\bar{w}(Z_{n+1})\big)\cdot 
\One{(Z_{n+1},\Zbag)\in A}}\bigg\}\\
&\leq \frac{1}{n}\sum_{i=1}^n \Ep{\Ptr_1\times\dots\times\Ptr_n\times\Ptr_i}
{\big(w^*_i(Z_{n+1})-\bar{w}(Z_{n+1})\big)_+} \\
&=\frac{1}{n}\sum_{i=1}^n \Ep{\Ptr_i}{\big(w^*_i(X,Y) - \bar{w}(X,Y)\big)_+}, 
\end{align*}
which completes our proof of the bound \eqref{eqn:wcp_drift_term_n+1}.
\end{document}